\documentclass{amsart}

\usepackage{verbatim}

\newcommand{\bC}{\mathbb{C}}

\newcommand{\bR}{\mathbb{R}}

\newcommand{\cC}{\mathcal{C}}
\newcommand{\cD}{\mathcal{D}}
\newcommand{\cE}{\mathcal{E}}

\newtheorem{theorem}{Theorem}[section]
\newtheorem{lemma}[theorem]{Lemma}
\newtheorem{proposition}[theorem]{Proposition}
\newtheorem{corollary}[theorem]{Corollary}

\theoremstyle{definition}
\newtheorem{definition}[theorem]{Definition}

\newtheorem{claim}[theorem]{Claim}

\theoremstyle{remark}
\newtheorem{remark}[theorem]{Remark}

\numberwithin{equation}{section}

\begin{document}

\title[Monge-Amp\`ere equation on compact Hermitian manifolds]{The complex Monge-Amp\`ere type equation on compact Hermitian manifolds and Applications}


\author{Ngoc Cuong Nguyen}
\address{Faculty of Mathematics and Computer Science, Jagiellonian University 30-348 Krak\'ow, \L ojasiewicza 6, Poland.}
\email{Nguyen.Ngoc.Cuong@im.uj.edu.pl}






\begin{abstract}
We prove the existence and uniqueness of continuous solutions to the complex Monge-Amp\`ere type equation with the right hand side in $L^p$, $p>1$, on compact Hermitian manifolds. Next, we generalise results of Eyssidieux, Guedj and Zeriahi \cite{EGZ09, EGZ11} to compact Hermitian manifolds which {\em a priori} are not  in the Fujiki class. These generalisations lead to a number of applications: we obtain partial results on a conjecture of Tosatti and Weinkove \cite{TW12a} and on a weak form of a conjecture of Demailly and Paun \cite{DP04}. 
\end{abstract}

\maketitle

\bigskip

\section*{Introduction}
\label{Sect-0}

Let $(X, \omega)$ be a $n$-dimensional compact Hermitian manifold. Given a non-negative function $f \in L^p(X, \omega^n)$, $p>1$, consider the complex Monge-Amp\`ere type equation 
\begin{equation}
\label{S0-eq1}
\begin{aligned}
	(\omega + dd^c \varphi)^n = e^{\lambda \varphi} f \omega^n \\
	\omega + dd^c \varphi \geq 0, \quad \lambda \geq 0,
\end{aligned}
\end{equation}
for real-valued function $\varphi$, where  
$
	d^c = \frac{i}{2\pi} (\bar \partial - \partial), \quad 
	dd^c =\frac{i}{\pi} \partial \bar\partial.
$

In the smooth category, for $\lambda>0$, given $f$ positive and smooth Cherrier \cite{cherrier87} proved the existence and uniqueness of the smooth admissible solution. Later on,  more general results are proved in this case by Hanani \cite{hanani96a, hanani96b}. 

The problem is considerably more difficult for the case $\lambda=0$.  At that time though many important {\em a priori} estimates were derived but the uniform  estimate was still missing  for the continuity argument. Under various extra assumptions the existence of a smooth admissible solution was obtained in \cite{cherrier87}.  Recently, the topic has been revived in the works of Guan-Li \cite{GL10} and Tosatti-Weinkove \cite{TW10a, TW10b}. Ultimately, 
the existence was proved, in full generality, by providing the uniform estimate,   by Tosatti and Weinkove \cite{TW10b}. Since then, the complex Monge-Amp\`ere equation,  both the elliptic and parabolic (the Chern-Ricci flow) version,  on compact Hermitian manifold was studied extensively  (see \cite{nie13}, \cite{sun13a, sun13b}, \cite{TW12b,TW12c, TWYang13}, \cite{ZZ11}). 

On the other hand, 
by developing pluri-potential estimates on compact Hermitian manifolds Dinew and Ko\l odziej \cite{DK12} were able to give another proof of the uniform estimate which was inspired by \cite{kolodziej98}. This proof allows the right hand side  in $L^p$, $p>1$, which has had many applications in K\"ahler geometry (see e.g. \cite{EGZ09}, \cite{pss12}). This motivates the development of    weak solutions theory for the equation \eqref{S0-eq1} on compact Hermitian manifolds.  Some geometric motivations and potential applications of such an investigation are discussed in \cite{dinew14}. 

As a first step, for $\lambda=0$, the existence of a continuous solution to the equation \eqref{S0-eq1} was proved in \cite{cuong-kolodziej13}. In this case the new technical tool is needed. We wish to exploit further the results in \cite{cuong-kolodziej13} to study the complex Monge-Amp\`ere type equation. It may help to understand deeper some open problem in Hermitian geometry.

For the case $\lambda>0$, our first result is an extension of a theorem of Cherrier \cite{cherrier87} to the degenerate right hand side.

\begin{theorem}
\label{S0-theo1}
Let $(X, \omega)$ be a $n$-dimensional compact Hermitian manifold. Let $0\leq f \in L^p(X, \omega^n)$, $p>1$, be such that $\int_X f \omega^n >0$. Fix $\lambda>0$. Then, there exist a unique continuous real-valued function $\varphi$ on $X$ satisfying
\begin{equation*}
\label{S0-eq2}
	(\omega + dd^c \varphi)^n = e^{\lambda \varphi} f \omega^n, \quad 
	\omega + dd^c \varphi \geq 0,
\end{equation*}
in the weak sense of currents.
\end{theorem}

Notice that we do not need to multiply the right hand side with a suitable constant as we do in the case $\lambda =0$ (see \cite{TW10b}, \cite{cuong-kolodziej13}). The following result will shed light on that mysterious constant.

\begin{corollary}
\label{S0-cor1} Under assumptions of Theorem~\ref{S0-theo1}.
For each $0< \varepsilon \leq 1$ let $\varphi_\varepsilon$ denote the unique continuous function  solving
\begin{equation*}
\label{S0-eq3}
	(\omega + dd^c \varphi_\varepsilon)^n = e^{\varepsilon \varphi_\varepsilon} f \omega^n,  \quad \omega + dd^c \varphi_\varepsilon \geq 0.
\end{equation*}
Suppose that  a continuous real-valued function $\varphi$ on $X$ and a constant $c>0$ solve 
\[
	(\omega + dd^c \varphi)^n = c f \omega^n, \quad
	\omega + dd^c \varphi \geq 0.
\]
Then, for any fixed $x\in X$,
\[
	c = \lim_{\varepsilon \to 0} e^{\varepsilon \varphi_\varepsilon(x)}.
\]
In particular, the constant $c$ in \cite[Theorem 0.1]{cuong-kolodziej13} is uniquely defined.
\end{corollary}

Tsuji \cite{tsuji88} and Tian-Zhang \cite{tian-zhang06} studied the K\"ahler-Ricci flow (parabolic Monge-Amp\`ere equation) and 
 its degenerate form on minimal algebraic varieties of general type. Later, the "finite energy" approach was introduced by Eyssidieux-Guedj-Zeriahi \cite{GZ05, GZ07, EGZ09} (see also \cite{dinew-zhang10}) building on the seminal works of Cegrell and Ko\l odziej \cite{cegrell98}, \cite{kolodziej98, kolodziej03}. Further advances in  \cite{EGZ09} or \cite{EGZ11} require the manifold to be at least in the Fujiki class, i.e. bimeromorphic to a K\"ahler manifold. Using recent results in \cite{DK12}, \cite{cuong-kolodziej13}, we are able to relax those assumptions. 

\begin{theorem}
\label{S0-theo2}
Let $(X, \omega)$ be a $n$-dimensional compact Hermitian manifold. Assume that $\beta\geq 0$  (semi-positive) is a smooth closed $(1,1)$-form on $X$  such that $\int_X \beta^n> 0$. Let $0\leq f \in L^p(X, \omega^n)$, $p>1$, be such that $\int_X f \omega^n>0$. Then, there exists a unique continuous real-valued function $\varphi$ solving
\begin{equation}
\label{S0-eq6}
	(\beta + dd^c \varphi)^n = e^\varphi f\omega^n, \quad
	\beta + dd^c \varphi \geq 0
\end{equation}
in the weak sense of currents.
\end{theorem}

The key ingredient to prove Theorem~\ref{S0-theo2} is a very precise uniform estimate for the complex Monge-Amp\`ere equation on compact Hermitian manifolds given in \cite{cuong-kolodziej13}. Once  the Monge-Amp\`ere type equation \eqref{S0-eq6} is solvable, we can adapt  arguments \cite{kolodziej03} to get the stability estimates (see \cite{dinew-zhang10}, \cite{EGZ09, EGZ11}). We get the following 

\begin{corollary}
\label{S0-cor2}
Under assumptions of Theorem~\ref{S0-theo2}. Assume that 
\[
	\int_X \beta^n  = \int_X f \omega^n.
\]
Then, there exists a unique continuous real-valued function $\varphi$, $\sup_X \varphi =0$, solving
\[
	(\beta + dd^c \varphi)^n = f\omega^n, \quad
	\beta + dd^c \varphi \geq 0.
\]
\end{corollary}

Notice that in both results above we still need to invoke the viscosity approach due to Eyssidieux-Guedj-Zeriahi \cite{EGZ11} for the continuity of solutions. 

\begin{remark}
So far the conjecture \cite[Conjecture 0.8]{DP04} is unsolved, so our results are of interest. If the conjecture was true, under the assumptions of Theorem~\ref{S0-theo2} and Corollary~\ref{S0-cor2},  the manifold would belong to the Fujiki class. 
\end{remark}

The solvability of the complex Monge-Amp\`ere equation with the metric being semi-positive gives a number of applications. First, we give a partial verification of a conjecture of Tosatti-Weinkove \cite{TW12a}.

\begin{theorem}
\label{S0-tw-conjecture}
Let $X$ be a $n$-dimensional compact complex manifold. Suppose there exists a 
class $\{\beta\} \in H^{1,1}_{BC} (X, \bR)$ which is semi-positive and satisfies 
$\int_X \beta^n >0$. Let $x_1, ..., x_N \in X$ be fixed points and let
$\tau_1, ..., \tau_N$ be positive real numbers so that 
\begin{equation}
\label{tw-assumption}
	\sum_{j=1}^N \tau_j^n < \int_X \beta^n.
\end{equation}
Then there exists a $\beta$-plurisubharmonic function $\varphi$ with logarithmic
poles at $x_1,..., x_N$:
\[
	\varphi(z) \leq \tau_j \log |z| + O(1),
\]
in a coordinate neighbourhood $(z_1, ..., z_n)$ centered at $x_j$, where 
$|z|^2 = |z_1|^2+...+|z_n|^2$.
\end{theorem}

Here we require the class $\{\beta\}$ to be semi-positive which is stronger than {\em nef}  in the  conjecture (see \cite{TW12a}). Tosatti and Weinkove proved their conjecture for $n=2,3$ and they proved for general $n\geq 4$ under the different assumptions: $X$ is Moishezon and $\{\beta\}$ is rational.

The second application is a partial result on the weak form of a conjecture of Demailly and Paun \cite{DP04}. 
Recently, Chiose \cite{chiose13} and Popovici \cite{popovici14} simplified some arguments in \cite{DP04}. Thanks to these simplifications and our results above we get the following.

\begin{theorem}
\label{S0-dp-semipositive}
Let $(X, \omega)$ be a $n$-dimensional compact complex manifold equipped with the pluriclosed metric $\omega$, i.e. $dd^c \omega =0$. Assume that $\{\beta\} \in H^{1,1}_{BC} (X, \bR)$ is a semi-positive cohomology class satisfying  
$\int_X \beta^n > 0$. Then $\{\beta\}$ contains a K\"ahler current $T$, i.e. $T\geq \delta \omega$ for some $\delta >0$.
\end{theorem}

If we can remove the pluriclosed assumption on $\omega$ in the theorem, then we would get the weak form of \cite[Conjecture 0.8]{DP04}. So, our result supports the affirmative answer of this conjecture.

The last application we wish to address is an improvement of a result of Gill \cite{gill13}. It is related to the Chern-Ricci flow on smooth models of general type.  
It is proved in \cite{dinew-zhang10}, \cite{EGZ09} that we may use the elliptic Monge-Amp\`ere equation to get better regularity of solutions that are obtained by the K\"ahler-Ricci flow. We can do the same for the complex Monge-Amp\`ere equation on compact Hermitian manifold and the Chern-Ricci flow (see Section~\ref{S4-2} for the details).

\bigskip

The organisation of the note is as follows. In Section~\ref{Sect-1} we give several estimates for the Monge-Amp\`ere operator.
These estimates are non-trivial extensions from the K\"ahler setting to the Hermitian one. In Section~\ref{Sect-2} we prove  Theorem~\ref{S0-theo1} and Corollary~\ref{S0-cor1}. Section~\ref{Sect-3} deals with the degenerate Monge-Amp\`ere equations, where Theorem~\ref{S0-theo2} and Corollary~\ref{S0-cor2} are proved by  using extensively results in Sections~\ref{Sect-1} and \ref{Sect-2}. 
Section~\ref{Sect-4} is devoted to applications: Theorems~\ref{S0-tw-conjecture}, \ref{S0-dp-semipositive} and the regularity of the Chern-Ricci flow.

\bigskip
\bigskip

{\bf Acknowledgement.} I am deeply grateful to S\l awomir Ko\l odziej for many inspiring discussions on the subject and encouragement me to write down this paper. It is improved significantly thanks to his thorough reading and editing. Part of this work was done during my visit at Toulouse Mathematics Institute in February-June 2014, where I  benefited a lot from many stimulating discussions with Ahmed Zeriahi. I would like to thank him for the invitation and his kind help during my stay there. I also thank the members of the institute for their hospitality. The visit was financially supported by the International Ph.D Program {\em" Geometry and Topology in Physical Models."} 

Before this paper was completed, I had given talks on its topic on the seminar {\em "Complex analysis and elliptic PDE's"} at Krak\'ow. I thank the members of the seminar for their interests. Especially, I would like thank S\l awomir Dinew for his useful comments and for sending me  his preprint \cite{dinew14}.
 The work was supported by NCN grant 2013/08/A/ST1/00312.


\bigskip

\section{Pluripotential estimates on compact Hermitian manifolds}
\label{Sect-1}

In this section we give several pluripotential estimates which are generalizations  from the case of compact K\"ahler manifolds to compact Hermitian manifolds. Because of the non closedness of metrics the proofs are often more complicated than their counterparts in the K\"ahler setting. However, the generalised formulations often keep the same spirit, up to the extra terms, as the original forms.

Let $(X,\omega)$ be a $n$-dimensional compact Hermitian manifold. 
The class of $\omega$-plurisubharmonic ($\omega$-psh) function is defined in the same way  as on the K\"aher manifold. A function $u : X \rightarrow [-\infty, +\infty[$ is $\omega$-psh if it is upper semi-continuous,  $u\in L^1(X, \omega^n)$ and 
\[
	 \omega + dd^c u \geq 0.
\]
The set of all $\omega$-psh functions on $X$ is denoted by $PSH(\omega)$. 
Bounded $\omega$-psh functions have most properties  as in the case $\omega$ is K\"ahler, though now $\omega$ does not have a  local potential. By using linear algebra we can define the Monge-Amp\`ere operator for bounded $\omega$-psh function. Hence, the notion of $cap_\omega$ and the Bedford-Taylor convergence theorem etc.... hold true. We refer the reader to \cite{DP04}, \cite{DK12}, \cite{cuong-kolodziej13}, \cite{dinew14} for more basic properties of $\omega$-psh functions.

Let us fix some notations that we will use throughout the note.
The "curvature" constant of the metric $\omega$ is denoted by $B= B(\omega)>0$ and it  satisfies 
\begin{equation}
\label{curvature}
	-B  \omega^2 \leq 2n dd^c \omega \leq B  \omega^2, \quad
	-B \omega^3 \leq 4n^2 d\omega \wedge d^c \omega \leq B\omega^3.
\end{equation}
The positive constants $C, C'=C(X,\omega)>0$ will appear at many places, we simply
call them to be uniform constants. Unless otherwise  stated, they may differ from place to place. When the volume form $\omega^n$ is clear in the context we denote for $r>0$, 
\[
	\|.\|_r = \left(\int_X |.|^r \omega^n \right)^\frac{1}{r} \quad\mbox{and}\quad
	\|.\|_\infty = \sup_X |.|.
\]
We often write 
$
	\omega_\varphi:= \omega + dd^c \varphi
$
for $\varphi \in PSH(\omega)$ 
and $L^r(\omega^n):=L^r(X,\omega^n)$.

The Chern - Levine - Nirenberg type inequality with $\omega$ being 
a Hermitian metric is as follows.

\begin{proposition}[CLN inequality]
\label{CLN}
Let $\psi, \varphi \in PSH(\omega)$ be such that $\sup_X \psi = 0$ and 
$0 \leq \varphi \leq 1$. Then, we have
\[
	\int_X |\psi| \omega_\varphi^n \leq C
\]
with a uniform constant $C>0$.
\end{proposition}

\begin{remark}
\label{cap-rk}
It also holds \cite[Proposition 2.3]{DK12} that
for every $\varphi_1, ..., \varphi_n \in PSH(\omega)$ and $0\leq \varphi_1, ..., \varphi_n \leq 1$,
\[
	\int_X \omega_{\varphi_1} \wedge ... \wedge \omega_{\varphi_n} \leq C
\]
where $C>0$ is a uniform constant. We also call those  CLN  inequalities. If we take $\varphi_1=...=\varphi_n =\varphi \in PSH(\omega)$, then CLN inequality shows that the notion of capacity \cite{BT82, kolodziej05} for a Hermitian metric $\omega$ on $X$ makes sense. Let $E \subset X$ be a Borel set, following \cite{kolodziej05, DK12} we denote
\begin{equation}
\label{cap-def}
	cap_\omega(E) := \sup\left\{\int_E (\omega+ dd^c \rho)^n : \rho \in PSH(\omega), \, 0 \leq \rho \leq 1 \right\}.
\end{equation}
We refer the reader to \cite{DK12}, \cite{dinew14} for basic properties of this capacity. One should keep in mind that because of the positivity of $\omega$ this capacity is comparable with the local Bedford-Taylor capacity $cap_\omega'(E)$ \cite[p. 52]{kolodziej05} (see also \cite[Proposition 3.10]{GZ05}).
\end{remark}

\begin{proof}[Proof of Proposition~\ref{CLN}]
It is similar to the case of $\varphi =0$ \cite[Proposition 2.1]{DK12}.  Let $\{B_j(s)\}_{j \in J}$ be a finite covering
of $X$, where $B_j(s) = B(x_j, s)$ is the ball  centered at $x_j$ of radius $s>0$. 
We may choose $s>0$ small enough such that for every
$j \in J$ there exists a smooth negative function $\rho_j$ on $B_j(3s)$ satisfying
\[
	\omega \leq dd^c \rho_j \quad \mbox{on} \quad B_j(2s).
\]
Thus, on $B_j(2s)$,
\[
	\omega_\varphi^n \leq
	\left[dd^c (\rho_j + \varphi) \right]^n.
\]
Then
\[
	\int_X |\psi| \omega_\varphi^n \leq 
	\sum_{j\in J} \int_{B_j(s)} |\psi| \omega_\varphi^n \leq
	\sum_{j\in J} \int_{B_j(s)} |\psi + \rho_j| (dd^c u_j)^n,
\]
where $u_j = \rho_j + \varphi, \psi + \rho_j$ belong to  $PSH(B_j(2s))$. Now, 
we are going to estimate from above for 
each term on the right hand side. By the $L^1-$ CLN inequality \cite[Chap. 3, Pro. 3.11]{demailly-book09} we have
\[
	\int_{B_j(s)} |\psi + \rho_j| (dd^c u_j)^n \leq 
	C(s) \| u_j\|_{L^\infty(B_j(2s))}^n \int_{B_j(2s)} |\psi + \rho_j| \omega^n.
\]
Since $0\leq \varphi \leq 1$ and $\rho_j$ is bounded on $B_j(2s)$,  it follows that
\[
	\int_{B_j(s)} |\psi + \rho_j| (dd^c u_j)^n 
	\leq C'(s) \left( \int_X |\psi| \omega^n + \|\rho_j\|_{L^\infty(B_j(2s))} \int_X \omega^n \right),
\]
where $C'(s)$ is independent of $\psi$. By \cite[Proposition 2.1]{DK12} 
the right hand side is uniformly bounded. Finally, since $J$ is a finite set, the proof follows.
\end{proof}

Thanks to the existence of continuous solutions to the complex Monge-Amp\`ere
equation with the right hand side  in $L^p(\omega^n)$, $p>1$, we get the following
the $L^1$-uniform bound. 
It seems that we may get the
similar statement as in \cite[ Proposition 2.7]{GZ05}, but the form given below is sufficient for our applications.

\begin{corollary}
\label{l1-uni-bound}
Let $0\leq F \in L^p(\omega^n)$, $p>1$, be such that $\int_X F \omega^n>0$.  
Then, there exists a uniform constant $C = C(X, \omega, \|F\|_p)>0$ 
such that for any $\psi \in PSH(\omega)$ with $\sup_X \psi =0$,
\[
	\int_X |\psi| F \omega^n <C.
\]
\end{corollary}

\begin{proof}
Using \cite[Theorem 0.1]{cuong-kolodziej13}, we solve $u\in PSH(\omega) \cap C(X)$, $\sup_X u =0$, and $c>0$ satisfying
\[
	(\omega + dd^c u)^n = c F \omega^n.
\]
Moreover, by \cite[Lemma 5.9]{cuong-kolodziej13} there exists $c_0 = C(\|F\|_p, \omega, X)>0$  such that
\begin{equation}
\label{cst-a-b}
	c_0 < c < \frac{1}{c_0}.
\end{equation}
According \cite[Corollary 5.6]{cuong-kolodziej13}, there exists $H = H(c_0, \|F\|_p, X, \omega)>0$ such that
\[
	- H \leq u \leq 0. 
\]
Therefore, it is not difficult to see the corollary follows from Proposition~\ref{CLN}
with the constant $C>0$ also depending on $H$. 
\end{proof}

The next result will be a version of the Cauchy-Schwarz inequality on compact
Hermitian manifolds. It has been useful in \cite{TW10b}.
The statement is similar to the classical one, however there appears a uniform constant to compensate for the torsion of $\omega$.

\begin{proposition}[Cauchy-Schwarz inequality]
\label{C-S-H}
Let $T$ be a positive current of bidegree $(n-2, n-2)$ of the form $\omega_{v_1} \wedge \omega_{v_2} \wedge ...\wedge \omega_{v_{n-2}}$, where $v_1,...,v_{n-2} \in PSH(\omega) \cap L^\infty(X)$. 
Then, for any $u\in PSH(\omega) \cap L^\infty(X)$,
\[
	\left| \int_X d u \wedge d^c \omega \wedge T \right|
\leq C\left ( \int_X du \wedge d^c u \wedge \omega \wedge T \right)^\frac{1}{2}
	\left(\int_X\omega^2 \wedge T \right)^\frac{1}{2} ,
\]
where $0< C = C(X, \omega)$ is a uniform constant.
\end{proposition}

\begin{proof} 
By partition of unity (a finite covering), it is enough to show that for 
the unit ball 
$U \subset \bC^n$ we have
\[
	\left|	\int_U d u \wedge d^c \omega \wedge T \right|
\leq	C\left ( \int_U du \wedge d^c u \wedge \omega \wedge T \right)^\frac{1}{2}
	\left(	\int_U\omega^2 \wedge T \right)^\frac{1}{2}.
\]
Hence, the proof is local. 
We can write 
\[
	\omega =i \sum_{j,k} a_{j k} (z)  dz_j \wedge d\bar z_k \quad \mbox{on } U
\]
where $a_{jk}(z) \in C^\infty(\overline{U})$.
Hence,
\[
	d \omega = i\sum_{j,k} da_{jk}(z) \wedge  dz_j \wedge d\bar z_k, \quad
	d^c \omega = i\sum_{j,k} d^ca_{jk}(z) \wedge  dz_j \wedge d\bar z_k.
\]
The polarisation of $dz_j \wedge d\bar z_k $ has the form
\begin{align*}
4 dz_j \wedge d\bar z_k 
&= d(z_j + z_k) \wedge d(\overline{ z_j + z_k}) 
	- d(z_j-z_k) \wedge d(\overline{z_j-z_k})		\\
&\quad +	id(z_j + i z_k) \wedge d(\overline{z_j + i z_k}) 
	- i d(z_j - i z_k) \wedge d(\overline{z_j - i z_k}),	\\
&:= \eta_1\wedge \bar \eta_1 - \eta_2\wedge \bar \eta_2 + i \eta_3\wedge \bar \eta_3 - i \eta_4\wedge \bar \eta_4 .
\end{align*}
It follows that
\begin{equation*}
\label{cs-ine-e1}
\left| \int_U du \wedge d^c a_{jk} \wedge  idz_j \wedge d\bar z_k \wedge T \right|
\leq		\sum_{m=1}^4 \left| \int_U d u\wedge d^c a_{jk} \wedge i \eta_m \wedge \bar \eta_m \wedge T \right|.
\end{equation*}
Now we can use the classical Cauchy-Schwarz inequality (see e.g. \cite[Schwarz's inequality p.7]{kolodziej05}) to each term of the right hand side. Let  us denote by $\eta$ one of the forms: 
$\eta_1, \eta_2, \eta_3, \eta_4$. Then,
\begin{align*}
 \left| \int_U d u\wedge d^c a_{jk} \wedge i \eta \wedge \bar \eta \wedge T \right| 
\leq 	
&\left( \int_U du \wedge d^c u \wedge i \eta \wedge \bar \eta \wedge T \right)^\frac{1}{2} \times\\
&\times \left( \int_U d a_{jk} \wedge d^c a_{jk} \wedge i \eta \wedge \bar \eta \wedge T \right)^\frac{1}{2}.	
\end{align*}
Observe that $0 \leq i \eta \wedge \bar \eta \leq C \omega$ and
$0 \leq d a_{jk} \wedge d^c a_{jk} \wedge \eta \wedge \bar \eta \leq C \omega^2$
on $\bar U$ for some uniform $C>0$. Hence,
\[
\int_U du \wedge d^c u \wedge i \eta \wedge \bar \eta \wedge T
\leq	C \int_U du \wedge d^c u \wedge \omega \wedge T
\]
and 
\[
	\int_U d a_{jk} \wedge d^c a_{jk} \wedge i \eta \wedge \bar \eta \wedge T
\leq 	C \int_U \omega^2 \wedge T.
\]
Altogether we get that
\begin{align*}
\left| \int_U d u \wedge d^c a_{jk} \wedge  idz_j \wedge d\bar z_k \wedge T \right|
&\leq  C\left(\int_U du \wedge d^c u \wedge \omega \wedge T \right)^\frac{1}{2} 
	\left( \int_U \omega^2 \wedge T \right)^\frac{1}{2} \\
&\leq  C\left(\int_X du \wedge d^c u \wedge \omega \wedge T \right)^\frac{1}{2} 
	\left( \int_X \omega^2 \wedge T \right)^\frac{1}{2}.
\end{align*}
Thus, the lemma follows.
\end{proof}

If $d\omega = 0$, then by the Stokes theorem, 
for any $\rho \in PSH(\omega)\cap L^\infty(X)$,
\begin{equation}
\label{vol-invar}
	\int_X \omega_\rho^n = \int_X\omega^n.
\end{equation}
This will be no longer true for a general Hermitian metric $\omega$ and the counter-example can be easily found. In fact, for compact complex surfaces, i.e  $n=2$, the inequality \eqref{vol-invar} holds if and only if $dd^c \omega =0$ (i.e. $\omega$ is a Gauduchon metric) by Lemma~\ref{la99}. 


From the potential theoretic point of view
it is often important to get bounds for the total mass of Monge-Amp\`ere operators or the Monge-Amp\`ere energy of $\omega$-psh functions. Below
we will consider estimates of those kinds. As we will see they are more useful when the "curvature"constant of the considered
metric is small. For example, the metric $\beta + \varepsilon \omega$, where $0< \varepsilon <<1$ and $\beta$ is a smooth closed semi-positive $(1,1)$-form.

Let $\Omega$ be another Hermitian metric on $X$ such that $\Omega \leq C \omega$ for some
uniform $C>0$. Suppose that there exists 
$0< B_\Omega <1$  satisfying
\begin{equation}
\label{curvature-Omega}
	-B_\Omega \omega^2 \leq 2n  dd^c \Omega \leq B_\Omega \omega^2, \quad
	-B_\Omega \omega^3 \leq 4n^2 d\Omega \wedge d^c \Omega \leq B_\Omega \omega^3.
\end{equation}
Then, the total mass of the Monge-Amp\`ere operator of a bounded $\Omega$-psh function is the total mass of $\Omega^n$ modulo the uniform norm of that function multiplied the curvature constant.

\begin{proposition}
\label{estimate-mass} Suppose that the Hermitian metric $\Omega$ satisfies \eqref{curvature-Omega}.
Let $u \in PSH(\Omega) \cap L^\infty(X)$ be such that $\sup_X u =0$. Then,
\[
	\int_X \Omega^n  - B_\Omega (1+\|u\|_\infty)^n C
\leq	\int_X (\Omega + dd^c u)^n 
\leq 	\int_X \Omega^n + B_\Omega (1+\|u\|_\infty)^n C,
\]
where $C>0$ is a uniform constant.
\end{proposition}

\begin{proof}
The proof is only a simple application of the Stokes theorem and induction process, so we only sketch it. To simplify notation we write $\Omega_u:= \Omega + dd^c u$.
Now, we compute
\begin{align*}
	dd^c [\Omega_u^l \wedge \Omega^{n-l-1}]
&=	l dd^c \Omega \wedge \Omega_u^{l-1} \wedge \Omega^{n-l-1}	\\
&\quad	+ l(l-1) d\Omega \wedge d^c \Omega \wedge \Omega_u^{l-2} \wedge \Omega^{n-l-1}	\\
&\quad	+ l(n-l-1) d\Omega \wedge d^c \Omega \wedge \Omega_u^{l-1} \wedge \Omega^{n-l-2}	\\
&\quad	+(n-l-1) dd^c \Omega \wedge \Omega_u^l \wedge \Omega^{n-l-2}	\\
&\quad	+l (n-l-1) d\Omega \wedge d^c \Omega \wedge \Omega_u^{l-1} \wedge \Omega^{n-l-2}	\\
&\quad	+ (n-l-1)(n-l-2) d\Omega \wedge d^c \Omega \wedge \Omega_u^l \wedge \Omega^{n-l-3} ,
\end{align*}
where $0\leq l \leq n-1$.
Recall from \eqref{curvature-Omega} that  $0< B_\Omega<1$ and 
the metric $\Omega$ satisfies
\[
	-B_\Omega \omega^2 \leq 2n dd^c \Omega \leq B_\Omega \omega^2, \quad
	-B_\Omega \omega^3 \leq 4n^2 d\Omega \wedge d^c \Omega \leq B_\Omega \omega^3.
\]
It follows that there exists $C_n >0$ depending only on dimension such that
\begin{equation}
\label{below-above-curvature}
	- B_\Omega  T(u, \omega, \Omega, l) 
\leq 	dd^c [\Omega_u^l \wedge \Omega^{n-l-1}]
\leq	B_\Omega T(u, \omega, \Omega, l),
\end{equation}
where 
\[
T(u, \omega, \Omega, l)  =\left[
\begin{aligned}
&\omega^2 \wedge \Omega_u^{l-1} \wedge \Omega^{n-l-1} 
+	\omega^3 \wedge \Omega_u^{l-2} \wedge \Omega^{n-l-1}	\\
&+ 	\omega^3 \wedge \Omega_u^{l-1} \wedge \Omega^{n-l-2}
+	\omega^2 \wedge \Omega_u^{l} \wedge \Omega^{n-l-2}	\\
&+	\omega^3 \wedge \Omega_u^{l-1} \wedge \Omega^{n-l-2}	
+	\omega^3 \wedge \Omega_u^{l} \wedge \Omega^{n-l-3}
\end{aligned}
\right]
\]
with the convention that $\Omega^l= \Omega_u^l= 1$ for $l\leq 0$. 
The proof goes by induction. First we write
\[
	\int_X \Omega_u^n
=	\int_X \Omega^n
	+ \int_X dd^c u \wedge \Omega^{n-1}
	+...+
	\int_X dd^c u \wedge  (\Omega + dd^c u)^{n-1}.
\]
By Stokes' theorem we have, for $0 \leq l \leq n-1$,
\[
	\int_X dd^c u \wedge \Omega_u^l \wedge \Omega^{n-l-1}
=	\int_X u dd^c [\Omega_u^l \wedge \Omega^{n-k-1}].	
\]
Therefore, by \eqref{below-above-curvature} and $u\leq 0$,
\[
\int_X \Omega_u^n
\leq		\int_X \Omega^n + B_\Omega\|u\|_\infty  \int_X \sum_{l=0}^{n-2} T(u,\omega,\Omega,l).
\]
Similarly,
\[
	\int_X \Omega_u^n
\geq	\int_X \Omega^n - B_\Omega\|u\|_\infty  \int_X \sum_{l=0}^{n-2} T(u,\omega,\Omega,l).
\]
Observe that the highest power of $\Omega_u$ in each term of $T(u, \omega, \Omega, l)$ is less than $n-1$, i.e. this power decreased by $1$ after applying the 
Stokes theorem. Thus, it is  by the  induction hypothesis and the Stokes theorem that we can  justify that for $0\leq l \leq n-2$
\[
	\int_X T(u, \omega, \Omega, l) \leq (1+ \|u\|_\infty)^{n-1} C. 
\]
and thus the proposition follows.
\end{proof}

The next result is the comparison of the Monge-Amp\`ere energy of $\Omega$-psh
functions. If $\Omega$ is  K\"ahler, i.e. $B_\Omega =0$, it is named as the fundamental inequality in \cite[Lemma 2.3]{GZ07}. 

\begin{proposition}
\label{in-e1-class}
Suppose that $\Omega = \beta + \varepsilon \omega$ with $\beta$ being a
closed semi-positive $(1,1)$-form and $0< \varepsilon <1$.
Let $u, v\in  PSH(\Omega) \cap L^\infty(X)$ be such that $u \leq v\leq -1$. 
Then,
\[
	\int_X (-v) (\Omega + dd^c v)^n 
\leq 	2^n \int_X (-u) (\Omega + dd^c u)^n 
	+ B_\Omega  \|u\|_\infty^{2n} \|v\|_\infty^{n} C,
\]
where $B_\Omega = \varepsilon B$ is the constant in \eqref{curvature-Omega}.
\end{proposition}

To prove this proposition we need a Cauchy-Schwarz type inequality for the metric $\Omega$, which is an immediate consequence of Proposition~\ref{C-S-H}.

\begin{lemma} 
\label{C-S-H-2}
Let $T$ be a positive current of bidegree $(n-2, n-2)$ of the form $\omega_{v_1} \wedge \omega_{v_2} \wedge ...\wedge \omega_{v_{n-2}}$, where $v_1,...,v_{n-2} \in PSH(\omega) \cap L^\infty(X)$. 
 Let $\Omega := \beta + \varepsilon \omega$, where
$\beta$ is a semi-positive $(1,1)$-form and $0< \varepsilon<1$.
There exists a uniform constant $0< C$ (independent of $\varepsilon$) such that for any 
$u\in PSH(\Omega) \cap L^\infty(X)$
\[
	\left| \int_X d u \wedge d^c \Omega \wedge T \right|
\leq	B_\Omega C\left ( \int_X du \wedge d^c u \wedge \omega \wedge T 
+	\int_X\omega^2 \wedge T \right),
\]
where $B_\Omega = \varepsilon B$ is the constant satisfying \eqref{curvature-Omega}.
\end{lemma}

\begin{proof}
We only need to observe that
\[
	d^c \Omega = \varepsilon d^c \omega
\]
and $B>0$ is a fixed uniform constant. Then, a simple application of Proposition~\ref{C-S-H} will give us the desired inequality.
\end{proof}

We are going to prove Proposition~\ref{in-e1-class}.

\begin{proof}
 Since $1\leq -v \leq -u$, we can replace $-v$ by $-u$ right away, then use the Stokes theorem to interchange roles of $u$ and $v$ in the integral and so on. This is how it was done for compact K\"ahler manifolds. However,  $\omega$ now is not closed, each time one applies the Stokes theorem some extra terms appear. The Cauchy-Schwarz type inequality and Proposition~\ref{curvature-Omega} are used to estimate these terms.
 
We write $\Omega_v = \Omega + dd^c v$.
Then
\begin{equation}
\label{e1-class-eq0}
	\int_X - v \Omega_v \wedge \Omega^{n-1}
\leq 	\int_X -u \Omega_v \wedge \Omega^{n-1}	\\
=	\int_X -u   \Omega^{n} +
	\int_X -u dd^c v \wedge \Omega^{n-1}.
\end{equation}
First, we are going to show that
\begin{equation}
\label{e1-class-eq1}
\begin{aligned}
	 \int_X -u \Omega^{n}
 \leq \int_X &-u \Omega_u \wedge \Omega^{n-1}  + 	\\
& +	 B_\Omega \|u\|_\infty^2 C_n  \int_X \left( \omega^2 \wedge \Omega^{n-2}  
	+ \omega^3 \wedge \Omega^{n-3} \right)		
\end{aligned}
\end{equation}
and
\begin{equation}
\label{e1-class-eq2}
\begin{aligned}
&	 \int_X -u dd^c v \wedge \Omega^{n-1} \leq \int_X -u \Omega_u \wedge \Omega^{n-1}	+ 	 \\
&\quad \qquad + B_\Omega \|u\|_\infty^2 \|v\|_\infty 
C \int_X (\omega^2 \wedge \Omega^{n-2}	+ \omega^3 \wedge \Omega^{n-3} + \omega^4 \wedge \Omega^{n-4}).
\end{aligned}
\end{equation}
To prove \eqref{e1-class-eq1}, we write
\[
	\int_X -u \Omega^n 
=	\int_X -u \Omega_u \wedge \Omega^{n-1} 
	+ \int_X u dd^c u \wedge\Omega^{n-1}.
\]
By Stokes' theorem
\[
	\int_X -u dd^c u \wedge \Omega^{n-1}
=	\int_X du \wedge d^c u \wedge \Omega^{n-1}
+	 (n-1)	
	\int_X u du \wedge d^c \Omega \wedge \Omega^{n-2}.	
\]
Again,
\begin{align*}
&	2 \int_X u du \wedge d^c \Omega \wedge \Omega^{n-2}
=	\int_X d(-u)^2 \wedge d^c \Omega \wedge \Omega^{n-2}	\\
&=	\int_X - (-u)^2 dd^c \Omega \wedge \Omega^{n-2}
- 	(n-2)  \int_X (-u)^2 d\Omega \wedge d^c \Omega 
	\wedge \Omega^{n-3}	\\
&\geq 	- B_\Omega \|u\|_\infty^2 \int_X \left[\omega^2 \wedge \Omega^{n-2} + \omega^3 \wedge \Omega^{n-3}\right].
\end{align*}
So, as $\int_X du \wedge d^c u \wedge \Omega^{n-1} \geq 0$, 
we get that
\[
	\int_X -u dd^c u \wedge \Omega^{n-1} 
\geq	- \frac{ (n-1)}{2} B_\Omega \|u\|_\infty^2 
		 \int_X \left[ \omega^2 \wedge \Omega^{n-2} + (n-2)\omega^3 \wedge \Omega^{n-3} \right] .
\]
Thus, \eqref{e1-class-eq1} is established. 

We continue to prove 
\eqref{e1-class-eq2}. By the Stokes theorem
\begin{equation}
\label{sum-2}
\begin{aligned}
	\int_X -u dd^c v \wedge \Omega^{n-1}
&=	\int_X -v dd^c u \wedge \Omega^{n-1}
\\ 
&\quad -	2(n-1)  \int_X v du \wedge d^c \Omega \wedge 
		\Omega^{n-2}	\\
&\quad + (n-1)(n-2)\int_X -v u d\Omega \wedge d^c \Omega
		\wedge \Omega^{n-3}.
\end{aligned}
\end{equation}
We proceed to estimate the right hand side of \eqref{sum-2}. The first term is bounded as follows.
\begin{equation}
\label{first-term}
	\int_X -v dd^c u \wedge \Omega^{n-1} 
\leq 	\int_X -u \Omega_u \wedge \Omega^{n-1}.
\end{equation}
The last term, according to \eqref{curvature-Omega}, satisfies
\begin{equation}
\label{last-term}
	(n-1)(n-2) \int_X -v u d\Omega \wedge d^c \Omega \wedge \Omega^{n-3}
\leq 	B_\Omega \|u\|_\infty \|v\|_\infty \int_X \omega^3 \wedge \Omega^{n-3}.
\end{equation}
For the second term, we apply Lemma~\ref{C-S-H-2} to $T = -v \Omega^{n-2}$, then
we get
\begin{equation}
\label{second-term}
	\int_X -v d^cu \wedge d \Omega \wedge \Omega^{n-2} 
\leq B_\Omega C\left (
\begin{aligned}
 &\int_X -v du \wedge d^c u \wedge \omega \wedge \Omega^{n-2}+ \\
 &+ \int_X-v \omega^2 \wedge \Omega^{n-2} 
\end{aligned}
\right)
\end{equation}
where $C>0$ depends only on $X, \omega$.  
It follows that the right hand side of 
\eqref{second-term} is bounded by
\begin{equation}
\label{second-term-1}
	B_\Omega \|v\|_\infty C
	\left ( \int_X du \wedge d^c u \wedge \omega \wedge \Omega^{n-2} + \int_X\omega^2 \wedge \Omega^{n-2} \right).
\end{equation}
Moreover, since $2du\wedge d^c u = dd^c (-u)^2 - 2 u dd^c u$, we have
\begin{align*}
	\int_X du \wedge d^c u \wedge \omega \wedge \Omega^{n-2}
&= \int_X dd^c(-u)^2 \wedge \omega \wedge\Omega^{n-2}
	- 2\int_X udd^cu \wedge \omega \wedge \Omega^{n-2} \\
&=	\int_X (-u)^2 dd^c[\omega \wedge \Omega^{n-2}] 
	+ 2\int_X (-u) \Omega_u \wedge \omega \wedge \Omega^{n-2} \\
&\quad 	- 2 \int_X (-u) \Omega\wedge \omega \wedge \Omega^{n-2},
\end{align*}
where in the second equality we used the Stokes theorem.
From this, it is clear that 
\[
	\int_X du \wedge d^c u \wedge \omega \wedge \Omega^{n-2}
\leq	\|u\|_\infty^2 C \int_X (\omega^2 \wedge \Omega^{n-2}	+
		\omega^3 \wedge \Omega^{n-3} + \omega^4 \wedge \Omega^{n-4}).
\]
Combining this and \eqref{second-term-1}, then \eqref{second-term}, we obtain
\begin{equation}
\label{second-term-2}
\begin{aligned}
&	\int_X v d^cu \wedge d \Omega \wedge \Omega^{n-2}	
\leq \\
&\leq	B_\Omega \|u\|_\infty^2 \|v\|_\infty C \int_X (\omega^2 \wedge \Omega^{n-2}	+ \omega^3 \wedge \Omega^{n-3} + \omega^4 \wedge \Omega^{n-4}).
\end{aligned}
\end{equation}
Then,  \eqref{e1-class-eq2} follows from \eqref{first-term}, \eqref{last-term} and \eqref{second-term-2}. 

According to \eqref{e1-class-eq0}, \eqref{e1-class-eq1} and \eqref{e1-class-eq2}, we have
\begin{align*}
	\int_X (- v) \Omega_v \wedge \Omega^{n-1}
\leq		2 \int_X(- u) \Omega_u \wedge \Omega^{n-1}
	+  B_\Omega  \|u\|_\infty^2 \|v\|_\infty C.
\end{align*}
Because 
\[
	d\Omega_v = d\Omega, \quad d^c \Omega_v = d^c \Omega,
\]
we may replace $\Omega$ by $\Omega_v$ in all argument above. Then, we get
\begin{equation}
\label{inter-v-u}
	\int_X (-v) \Omega_v^n
\leq		2 \int_X (-u) \Omega_u \wedge \Omega_v^{n-1} 
	 +	B_\Omega  \|u\|_\infty^{2} \|v\|_\infty^{n} C,
\end{equation}
where the constant $\|v\|_\infty^{n-1}C>0$, by Proposition~\ref{estimate-mass}, 
is  the upper bound for
\[
	 \int_X (\omega^2 \wedge \Omega_v^{n-2}	
	 + \omega^3 \wedge \Omega_v^{n-3} 
	 + \omega^4 \wedge \Omega_v^{n-4})
\]
instead of $\int_X (\omega^2 \wedge \Omega^{n-2}	+ \omega^3 \wedge \Omega^{n-3} + \omega^4 \wedge \Omega^{n-4})$ on the right hand side of
\eqref{second-term-2}.
Similarly, 
we show by induction ($0<B_\Omega<1$ and $u,v \leq -1$) that
\begin{align*}
	\int_X \omega^k \wedge \Omega_ u^{n-k}
\leq	\|u\|_\infty^{n-k} C, \quad
	\int_X \omega^k \wedge \Omega_v^{n-k}
\leq	 \|v\|_\infty^{n-k} C.
\end{align*}
Thus, we have seen in \eqref{inter-v-u} that one term $\Omega_v$ is replaced by $\Omega_u$. We continue replacing $\Omega_v^{n-1}$ by
$\Omega_u \wedge \Omega_v^{n-2}$ and so on. Finally, 
\begin{align*}
	\int_X (-v) (\Omega + dd^c v)^n
\leq		2^n \int_X (-u) (\Omega + dd^c u)^n 
	 +	B_\Omega  \|u\|_\infty^{2n} \|v\|_\infty^{n}C.
\end{align*}
Thus, the proof is finished.
\end{proof}

On a general compact Hermitian manifold $X$ there may not exist a smooth closed semi-positive $(1,1)$ form $\beta$ such that $\int_X \beta^n >0$ and there are also non-K\"ahler manifolds possessing such forms (see e.g. \cite[Example 1.8, 1.5]{chiose13}). If such 
a form exists, then the notion $cap_\beta$ makes sense and more importantly the volume-capacity inequality (\cite{kolodziej98, kolodziej05}, \cite{EGZ09}, \cite{demailly-pali10}) still holds.

\begin{proposition}\cite[Lemma 2.9]{demailly-pali10}
\label{vol-cap-ine}
Let $(X, \omega)$ be a $n$-dimensional compact Hermitian manifold.
Assume that $\beta$ is a smooth closed semi-positive $(1,1)$-form on $X$ satisfying 
$\int_X \beta^n>0$. We define for any Borel set $E\subset X$,
\[
	cap_\beta(E) = \sup \left\{ \int_E (\beta + dd^c v)^n : v \in PSH(X, \beta), \quad
	0\leq v \leq 1\right\},
\]
where $PSH(X, \beta)$ is the set of all $\beta$-psh functions on $X$. 
Then, there exists uniform constant $a, C>0$ such that
\[
	Vol_\omega(E)= \int_E \omega^n \leq C \exp \left(-\frac{a}{cap_\beta^\frac{1}{n}(E)} \right).
\]
\end{proposition}

We end this section with the mixed type inequality in the Hermitian setting. We refer the reader to \cite{dinew09} for the most general form of this kind of inequality.
\begin{lemma}
\label{mixine}
Let $0 \leq f, g \in L^1(\omega^n)$ and $u, v \in PSH(\omega) \cap L^\infty(X)$. Suppose that $\omega_u^n \geq f \omega^n$ and $\omega_v^n  \geq g \omega^n$ on $X$. Then for $ k = 0, ...,  n$
\[
	\omega_u^k \wedge \omega_v^{n-k} 
		\geq f^{\frac{k}{n}} \, g^{\frac{n-k}{n}} \omega^n \quad \mbox{ on } \quad 
	X.
\]
In particular, for $0< \delta<1$,
\[
	\omega_{\delta u + (1-\delta) v}^n 
	\geq  \left[  \delta f^\frac{1}{n} 
			+ (1-\delta) g^\frac{1}{n} \right]^n \omega^n
	\quad \mbox{ on } \quad 
	X.
\]
\end{lemma}
\begin{proof} This is a local problem. The proof  is a consequence of
the solvability of the Dirichlet problem and the stability estimates for solutions to 
the Monge-Amp\`ere equation in a ball in $\bC^n$. Notice that the background form in the equation is the Hermitian form $\omega$. The results in \cite[Section 4]{cuong-kolodziej13} are enough for the proof as in \cite[Lemma 6.2]{kolodziej05}.
\end{proof}

\bigskip

\section{The complex Monge-Amp\`ere type equations}
\label{Sect-2}

Let $(X,\omega)$ be a $n$-dimensional compact Hermitian manifold.
In this section, we are going to study the weak solutions to the equation
\begin{equation}
\label{mae-type}
\begin{aligned}
&	\varphi \in PSH(\omega) \cap L^\infty(X), \\
&	(\omega + dd^c \varphi)^n = e^{\lambda \varphi} f \omega^n, \quad
	\lambda \geq 0,
\end{aligned}
\end{equation}
where $0\leq f \in L^p(\omega^n)$, $p>1$. 

The continuous solutions to the equation \eqref{mae-type} for $\lambda =0$ were recently obtained in \cite{cuong-kolodziej13} 
and we will use the results in \cite{cuong-kolodziej13} to study the case $\lambda>0$. The difference is that we get not only the existence but also the uniqueness of the continuous solution.

When $\lambda>0$, after a rescaling, we only need to consider $\lambda =1$,
i.e. the equation
\begin{equation} \label{n-ma}
	(\omega + dd^c \varphi)^n = e^\varphi f \omega^n
\end{equation}
where $\varphi \in PSH(\omega) \cap L^\infty(X)$ and $0\leq f \in L^p(\omega^n)$,
$p>1$.  

When $f>0$ and $f$ is smooth Cherrier \cite{cherrier87} proved that there exists unique smooth solution. Our result can be considered as an extension of his result for non-negative right hand side in $L^p(\omega^n)$, $p>1$.

\begin{theorem}
\label{existence-ma}
Let $0\leq f \in L^p(\omega^n)$, $p>1$, be such that $\int_X f \omega^n>0$, then the equation \eqref{n-ma} has a unique continuous solution.
\end{theorem}

\begin{remark}
The assumption $\int_X f \omega^n>0$ is also necessary to guarantee the existence of a bounded solution to the equation (see \cite[Remark 5.7]{cuong-kolodziej13}). 
\end{remark}

The uniqueness  is just a consequence of the following statement. 

\begin{lemma} 
\label{decreasing-property}
Suppose that $\varphi, \psi \in PSH(\omega) \cap L^\infty(X)$ satisfy
\[
	(\omega + dd^c \varphi)^n = e^\varphi f \omega^n, \quad
	(\omega + dd^c \psi)^n = e^\psi g \omega^n
\]
with $0 \leq f, g \in L^p(\omega^n)$, $p>1$.  If $f \leq g$, then $\psi \leq \varphi$. 
In particular, there is at most one 
function $\varphi \in PSH(\omega) \cap C(X)$ such that
\[
	(\omega + dd^c \varphi)^n = e^\varphi f \omega^n.
\]
\end{lemma}

\begin{proof}
We argue by contradiction. Suppose that $\{ \psi > \varphi\}$ is non-empty. Then,
$m= \inf_X (\varphi - \psi) <0$. Fix  $0< \varepsilon <<1$ to be determined later, 
and  denote by  $m(\varepsilon)= \inf_X [\varphi - (1-\varepsilon)\psi]$. It is clear that
\begin{equation*}
\label{dec-pro-e1}
	m - \varepsilon \| \psi \|_\infty \leq m(\varepsilon) \leq 
	m + \varepsilon \| \psi \|_\infty. 
\end{equation*}
Hence, on $U(\varepsilon, s):=\{ \varphi < (1-\varepsilon)\psi + m(\varepsilon) +s \}$ we have
\begin{equation}
\label{dec-pro-e2}
	\omega_{(1-\varepsilon)\psi}^n 
\geq		(1-\varepsilon)^n e^\psi g \omega^n 
\geq		(1-\varepsilon)^n e^{\varphi -m -2 \varepsilon \| \psi \|_\infty -s} g\omega^n . 
\end{equation}
The modified comparison principle \cite[Theorem 0.2]{cuong-kolodziej13} reads 
for $0< s < \varepsilon_0=\frac{\varepsilon^3}{16B}$,
\[
	\int_{\{ \varphi < (1- \varepsilon) \psi + m(\varepsilon) +s \}} 
	\omega_{(1-\varepsilon)\psi}^n 
\leq 	(1 + \frac{C s}{\varepsilon^n}) \int_{\{ \varphi < (1- \varepsilon)\psi + m(\varepsilon) +s \}} 
	\omega_\varphi^n 
\]
where $C>0$ is a uniform constant.
According to \eqref{dec-pro-e2} we get 
\[
	(1-\varepsilon)^n e^{-m- 2\varepsilon \|\psi\|_\infty - s}
	\int_{U(\varepsilon, s)}  e^\varphi f \omega^n 
\leq	(1 + \frac{C s}{\varepsilon^n}) \int_{U(\varepsilon, s)} e^\varphi f \omega^n.
\]
Note that for every $0< s < \varepsilon_0$
\[
	\int_{U(\varepsilon, s)} e^\varphi f \omega^n =	\int_{\{ \varphi < (1-\varepsilon)\psi + m(\varepsilon) +s \}} \omega_\varphi^n
	>0.
\]
Since $m <0$, we may choose $0< \varepsilon$ so small that 
\[
	(1-\varepsilon)^n e^{-\frac{m}{2} - 2\varepsilon \|\psi\|_\infty} > 1 + b 
\]
for some $b = b(m, \varepsilon)>0$.  Thus, we get for every $0< s <\min\{ \varepsilon_0,  - \frac{m}{2}\}$ that
\[
	0< b \leq \frac{C s}{\varepsilon^n}.
\]
It is impossible when $s>0$ is small enough. Thus, the proof follows.
\end{proof}

By examining the above proof, it is quite easy to get the following useful fact which is obvious in the K\"ahler case ($d\omega =0$).

\begin{corollary}
\label{const-comp}
If $u, v \in PSH(\omega) \cap L^\infty(X)$ satisfy
\[
	\omega_u^n \leq  c \; \omega_v^n
\]
for some $c>0$, then $c \geq 1$.
\end{corollary}

Before  proving the existence of a  continuous solution we first give an {\em a priori} estimate. This estimate will frequently be used in the sequel. 

\begin{proposition}
\label{sup-bound-general}
Let $(X,\omega)$ be a $n$-dimensional compact Hermitian manifold.
Let $G\in C^\infty(X)$ be the (Gauduchon) function such that  $e^G\omega^{n-1}$ is $dd^c$-closed. Let $0\leq g \in L^p(\omega)$, $p>1$, be such that $\int_X g\omega^n>0$. Let us denote 
\[
	\alpha_0 = \int_X e^G  \omega^n, \quad 
	A= \int_X g^\frac{1}{n} e^G\omega^n>0.
\]
Consider $c_g>0$ and $u\in PSH(\omega) \cap C(X)$  solving
 \begin{equation}
 \label{sbg-e0}
 	(\omega + dd^c u)^n = c_g \, g \omega^n
 \end{equation}
 (see \cite[Theorem 0.1]{cuong-kolodziej13} and \eqref{cst-a-b} in Section~\ref{Sect-1}).
Suppose that $v \in PSH(\omega) \cap L^\infty(X)$ satisfies
\[
	(\omega + dd^c v)^n = e^v g \omega^n.
\]
Then,
\[
	\log c_g \leq \sup_X v \leq C(\|g\|_p, A, X, \omega) + n \log \frac{\alpha_0}{A},
\]
where 
\[
C(\|g\|_p, A, X, \omega) = \int_X |v - \sup_X v| \frac{g^\frac{1}{n} e^G\omega^n}{A}.
\]
\end{proposition}

\begin{proof}
Put $M= \sup_X v$ and $v_0 = v-M$.
By the mixed type inequality (Lemma~\ref{mixine})
\[
	\omega_v \wedge \omega^{n-1} \geq e^\frac{v}{n}g^\frac{1}{n} \omega^n.
\]
Therefore,
\[
	\omega_v\wedge e^G\omega^{n-1} 
	\geq e^\frac{v}{n}g^\frac{1}{n} e^G\omega^{n} .
\]
By the Gauduchon condition it follows that 
\begin{equation}
\label{sbg-e1}
	\int_X e^G \omega^{n} 
	\geq \int_X e^\frac{v}{n} g^\frac{1}{n}e^G\omega^{n}.
\end{equation}
The Jensen inequality gives that
\begin{equation}
\label{sbg-e2}
	\log\left(\frac{1}{A} \int_X e^\frac{v}{n} g^\frac{1}{n} e^G \omega^{n} \right) \geq  \int_X \frac{v}{n} \frac{g^\frac{1}{n}e^G \omega^{n}}{A}.
\end{equation}
Combining \eqref{sbg-e1} and \eqref{sbg-e2} and $v= v_0 + M$ we get
\[
	n \log \frac{\alpha_0}{A} \geq M + \int_X v_0 \, \frac{g^\frac{1}{n}e^G\omega^n}{A}.
\]
Hence, we get the second inequality of Proposition~\ref{sup-bound-general}. It remains to prove the first one. According
to \eqref{sbg-e0} and $v= v_0 +M$ we get that
\[
	\omega_{v}^n = e^{v_0 +M} g \omega^n \leq e^M g \omega^n
	= \frac{e^M}{c_g} \, \omega_u^n.
\]
Therefore, the first inequality follows from Corollary~\ref{const-comp}.
\end{proof}

We are in the position to prove existence of continuous solution to Monge-Amp\`ere equation with the right hand side in $L^p$, $p>1$.

\begin{proof}[Proof of Theorem~\ref{existence-ma}]
Suppose that the sequence of smooth functions 
$f_j > 0$, $j\geq 1$, converges in $L^p(\omega^n)$ to $f$ as $j \to +\infty$. By a theorem of Cherrier \cite[Th\'eor\`em~1, p.373]{cherrier87}, there exists a unique $\varphi_j \in PSH(\omega)\cap C^\infty(X)$ such that
\begin{equation*}
	(\omega + dd^c \varphi_j)^n = e^{\varphi_j} f_j \omega^n.
\end{equation*}
Let us denote 
\[
	M_j = \sup_X \varphi_j , \quad \mbox{and} \quad \psi_j = \varphi_j -M_j \leq 0.
\]
Then, the above equation can be rewritten as follows.
\begin{equation}
\label{smooth-ma}
	(\omega + dd^c \psi_j)^n = e^{\psi_j + M_j} f_j \omega^n.
\end{equation}

\begin{claim}
\label{c-uniform-bound-M_j}
	$M_j$ is uniformly bounded.
\end{claim}
\begin{proof}[Proof of Claim~\ref{c-uniform-bound-M_j}] It's a consequence of
Proposition~\ref{sup-bound-general} and Corollary~\ref{l1-uni-bound}. Let us define
\[
	A_j = \int_X f_j^\frac{1}{n} e^G \omega^n, \quad
	\alpha_0 = \int_X e^G\omega^n .
\]
Applying the second inequality in Proposition~\ref{sup-bound-general} one gets  that
\begin{equation}
\label{bound-M_j}
	M_j \leq C(\|f_j\|_p, A_j, X, \omega)+ n  \log \frac{\alpha_0}{A_j}
\end{equation}
Thus, it is sufficient to show that 
\[
	\frac{1}{C}< A_j< C
\]
for some uniform constant $C>0$ independent of $j\geq 1$. Indeed,
it is clear that 
\[
	A_j \rightarrow \int_X f^\frac{1}{n} e^G\omega^n >0
\]
as $j \rightarrow +\infty$. Hence, $A_j$ is uniformly bounded from below away from $0$. 
Next, by the H\"older inequality we get that
\[
	 \int_X f_j^\frac{1}{n} e^G \omega^n \leq
	 \left( \int_X f_j e^G \omega^n \right)^\frac{1}{n} 
	 \left( \int_X e^G \omega^n\right)^\frac{n-1}{n}.
\]
Therefore, for $j$ large,
\[
	\frac{1}{2} \int_X f^\frac{1}{n} e^G \omega^n
\leq	A_j 
\leq 	\alpha^\frac{n-1}{n} \left( \int_X f_j e^G\omega^n \right)^\frac{1}{n}.
\]
We have obtained the upper bound part.  It remains to show $M_j$ is uniformly bounded  from below. We solve $u_j \in PSH(\omega) \cap C(X)$ 
\[
	(\omega + dd^c u_j)^n = c_j f_j \omega^n ,
\]
where $c_j> c_0:= c_0(\|f_j\|_p, X, \omega)>0$ (see \eqref{cst-a-b} in Corollary~\ref{l1-uni-bound}). The first inequality in Proposition~\ref{sup-bound-general} gives
\[
	\log c_0 \leq \log c_j \leq M_j.
\]
The claim is proven.
\end{proof}

Claim~\ref{c-uniform-bound-M_j} implies that
the right hand side of \eqref{smooth-ma} is uniformly bounded 
in $L^p(\omega^n)$. Since the family $\{\psi_j\}$ is compact in $L^1(\omega^n)$,
after passing to a subsequence, we may assume that it is a Cauchy sequence in 
$L^1(\omega^n)$, and suppose that $e^{M_j}$ converges to $e^M$.  
\cite[Corollary 5.10]{cuong-kolodziej13} shows that it is actually a Cauchy sequence in $C(X)$. Thus, $\psi_j$ converges uniformly to $\psi \in PSH(\omega)\cap C(X)$, $\sup_X \psi =0$. Hence $\varphi = \psi +M$ solves
\begin{equation}
	(\omega + dd^c \varphi)^n = e^{\varphi} f \omega^n. 
\end{equation}
Thus, we have finished the proof of existence.
\end{proof}

For $\lambda=0$ a smooth (or continuous) solution exists after multiplication of the right hand side by a suitable constant \cite[Corollary 1]{TW10b}, \cite[Theorem 0.1]{cuong-kolodziej13}. When $\lambda>0$ the difference is that the adjustive constant $c>0$ does not appear on the right hand side.  
Our next result shows how this constant can be computed using Theorem~\ref{existence-ma}.

\begin{corollary} 
\label{compute-constant}
Let $0\leq f \in L^p(\omega^n)$, $p>1$ be such that $\int_X f \omega^n >0$. Let $c>0$ and  $\varphi \in PSH(\omega) \cap C(X)$ solve
\begin{equation}
\label{epsilon-0-ma}
	(\omega + dd^c \varphi)^n = c f \omega^n.
\end{equation}
Let  $\varphi_\varepsilon \in PSH(\omega) \cap C(X)$, $0< \varepsilon \leq 1$, be the unique solution to
\begin{equation}
\label{epsilon-ma}
	(\omega + dd^c \varphi_\varepsilon)^n 
	= e^{\varepsilon \varphi_\varepsilon} f \omega^n.
\end{equation}
Then, for any fixed $x\in X$,
\[
	c = \lim_{\varepsilon \rightarrow 0} e^{\varepsilon \varphi_\varepsilon(x)}
	= \lim_{\varepsilon \rightarrow 0} e^{\varepsilon M_\varepsilon}
\]
where $M_\varepsilon = \sup_X\varphi_\varepsilon $.
\end{corollary}

\begin{remark}
In the case $\omega$ is K\"ahler, under the necessary condition for the solution of the equation \eqref{epsilon-0-ma}
\[
	\int_X f \omega^n= \int_X \omega^n,
\]
we are able to show that 
\[
	0 \leq M_\varepsilon \leq C.
\]
Therefore, 
$
	c = \lim_{\varepsilon \rightarrow 0} e^{\varepsilon M_\varepsilon}
	=1.
$
\end{remark}

\begin{proof}[Proof of Corollary~\ref{compute-constant}]
Put $M_\varepsilon = \sup_X \varphi_\varepsilon$ and $\psi_\varepsilon = \varphi_\varepsilon - M_\varepsilon$. The equation \eqref{epsilon-ma} is now rewritten as 
\begin{equation}
\label{e2-ma}
	(\omega + dd^c \psi_\varepsilon)^n 
=	e^{\epsilon \psi_\varepsilon + \varepsilon M_\varepsilon} f \omega^n.
\end{equation}
Again by Proposition~\ref{sup-bound-general} we get
\[
	\varepsilon M_\varepsilon \leq \varepsilon C + n \log \frac{\alpha_0}{A}
	< C' ,
\]
where $\alpha_0= \int_X e^G \omega^n$, $A = \int_X f^\frac{1}{n} e^G \omega^n  >0$ and $0<\varepsilon \leq 1$. 
Then, it follows from \eqref{e2-ma} that
\[
	\omega_{\psi_\varepsilon}^n \leq e^{\varepsilon M_\varepsilon} f \omega^n
	\leq e^{C'} f \omega^n.
\]
By \cite[Corollary 5.6]{cuong-kolodziej13} there exists a uniform constant $H = H(\|f\|_{L^p(\omega^n)}, \omega) >0$   such that
\begin{equation}
\label{uniform-oscilliation}
	-H \leq \psi_\varepsilon = \varphi_\varepsilon - M_\varepsilon \leq 0.
\end{equation}
Note that, by Corollary~\ref{const-comp}, we also have 
\begin{equation}
\label{lower-bound}
	e^{\varepsilon M_\varepsilon} \geq  c.
\end{equation}
The right hand side of \eqref{e2-ma} is uniformly bounded in $L^p(\omega^n)$,
and we may suppose that $\{\psi_\varepsilon\}$ is a Cauchy sequence in $L^1(\omega^n)$.  \cite[Corollary 5.10]{cuong-kolodziej13} implies that it is actually 
a Cauchy sequence in $C(X)$. Let us denote by $\psi$ the limit point, and suppose
that $\varepsilon M_\varepsilon \rightarrow M$. Taking limits of two sides we get that
\begin{equation}
\label{cst-c}
	(\omega + dd^c \psi)^n = e^M f \omega^n.
\end{equation}
It follows from Corollary~\ref{const-comp} and equations \eqref{epsilon-0-ma},  \eqref{cst-c} that
\[
	c = e^M= \lim_{\varepsilon \rightarrow 0} e^{\varepsilon M_\varepsilon}.
\]
Moreover, this equality and \eqref{uniform-oscilliation} imply, for any fixed $x\in X$, 
\[
	 \lim_{\varepsilon \rightarrow 0} e^{\varepsilon \varphi_\varepsilon(x)} =c.
\]
Thus, the proposition follows.
\end{proof}

\bigskip

\section{Degenerate complex Monge-Amp\`ere equations}
\label{Sect-3}

Let $(X,\omega)$ be a $n$-dimensional compact Hermitian manifold.
Suppose that there is a  smooth closed semi-positive  $(1,1)$-form $\beta$ on $X$ satisfying
\begin{equation}
\label{big-form}
	\int_X \beta^n >0, \quad \mbox{normalized by } 
	\int_X \beta^n =1.
\end{equation}
Our goal is to extend the results in \cite{EGZ09, EGZ11} to compact Hermitian manifolds which {\em a priori} do not belong to the Fujiki class. On the other hand, if the conjecture of Demailly and Paun \cite[Conjecture 0.8]{DP04} holds, then our results are just the consequence of those in \cite{EGZ09, EGZ11}.  We refer the reader to the survey \cite{pss12} for the state-of-the-art of results on the complex Monge-Amp\`ere equations on K\"ahler manifolds, in which many geometric problems and motivations to study the weak solutions are discussed.

We are interested in finding weak solutions to the following
equations
\[
	(\beta + dd^c \varphi)^n = e^{\lambda \varphi} f \omega^n, \quad
	\lambda\geq 0
\]
where $0\leq f \in L^p(\omega^n)$, $p>1$ and the function $\varphi$ belongs 
to
\[
	PSH(\beta) := \{ u \in L^1(X, \bR \cup {-\infty}): \beta + dd^c u \geq 0 
	\mbox{ and } u \mbox{ is u.s.c}\}.
\]
If a function belongs to $PSH(\beta)$ then we call it  a $\beta$-psh function. We refer to \cite{GZ05, GZ07, EGZ09, EGZ11} for properties of $\beta$-psh functions.

Our first result in this section deals with  the case $\lambda =1$ or equivalently  $\lambda>0$ (up to a rescaling).

\begin{theorem}
\label{existence-dmae}
Let $0\leq f \in L^p(\omega^n)$, $p>1$, be such that 
$
	\int_X f\omega^n>0
$
Then, there exists a unique function $\varphi \in PSH(\beta) \cap C(X)$ satisfying
\[
	(\beta + dd^c \varphi)^n = e^\varphi f \omega^n
\]
in the weak sense of currents.
\end{theorem}

Then, using this theorem we get the following result which includes Corollary~\ref{S0-cor2}.

\begin{theorem}
\label{existence-dmae-2}
Let $0\leq f \in L^p(\omega^n)$, $p>1$, be such that 
$
	\int_X f\omega^n>0.
$
Let us denote  by $\varphi_\varepsilon \in PSH(\beta) \cap C(X)$, $0<\varepsilon \leq 1$,  the unique continuous solution to the equation
\begin{equation}
\label{dmae-2-e1}
	(\beta + dd^c \varphi_\varepsilon)^n 
	= e^{\varepsilon \varphi_\varepsilon} f \omega^n.
\end{equation}
Then, for any fixed $x\in X$, 
\begin{equation}
\label{dmae-2-e2}
	e^{\varepsilon \varphi_\varepsilon(x)} \to 
	c := \frac{\int_X \beta^n}{\int_X f \omega^n}
	\quad \mbox{as} \quad \varepsilon \to 0.
\end{equation}
Moreover, as $\varepsilon \to 0$, $\varphi_\varepsilon - \sup_X \varphi_\varepsilon$ converges uniformly
to the unique continuous $\varphi \in PSH(\beta)$, $\sup_X \varphi =0$, satisfying
the equation
\begin{equation*}
\label{dmae-2-e3}
	(\beta + dd^c \varphi)^n = c \, f \omega^n
\end{equation*}
in the weak sense of currents.
\end{theorem}

\begin{remark}
\label{uniqueness-big-case}
The uniqueness  in Theorems~\ref{existence-dmae} and ~\ref{existence-dmae-2} (continuous or bounded solutions) is well-known. In fact, once we have the volume-capacity inequality (Proposition~\ref{vol-cap-ine}) it can be obtained by adapting the stability estimate of Ko\l odziej \cite{kolodziej03, kolodziej05} to this setting as it was done by Dinew and Zhang \cite[Corollary 1.2]{dinew-zhang10}. One special feature of those equations is that we do have the classical comparison principle for $\beta$-psh functions and the invariance of volume of their Monge-Amp\`ere measures, i.e.  $\int_X \beta_u^n = \int_X \beta^n$ for all $u \in PSH(\beta)\cap L^\infty(X)$. 
\end{remark}

\begin{remark}
\label{method-dmae} If $X$ possesses a K\"ahler metric or it  is in the  Fujiki class, then two theorems above are due to Eyssidieux-Guedj-Zeriahi \cite{EGZ09, EGZ11} who have generalised Ko\l odziej's results \cite{kolodziej98, kolodziej03} to the degenerate left hand side.
In the course of the proof of the theorems we follow the approach outlined in  
\cite{EGZ09, EGZ11}. The main difference is that while on K\"ahler manifolds one can choose a sequence of K\"ahler metrics which tend to the degenerate metric
$\beta$, it is no longer the case on compact Hermitian manifolds. We have to approximate 
the degenerate metric $\beta \geq 0$ by a sequence of Hermitian metrics, such as 
$\{\beta + \varepsilon \omega\}_{\varepsilon>0}$. In this case we have to deal with the torsion of $\omega$. The new tools are the recent results \cite{cuong-kolodziej13} which will help us carrying out the arguments. We shall see that the {\em modified} comparison principle \cite[Theorem 0.2]{cuong-kolodziej13} appears at almost every step.
\end{remark}

First we will prove Theorem~\ref{existence-dmae-2} assuming 
Theorem~\ref{existence-dmae}. The proof will make use of the stability estimate
of $L^\infty-L^r$, $r\geq 1$, of solutions  \cite[Theorem 4.5]{kolodziej03}, \cite[Proposition 3.3]{EGZ09} provided the right hand sides are in $L^p(\omega^n)$, $p>1$.

\begin{proof}[Proof of Theorem~\ref{existence-dmae-2}]
Let us denote $M_\varepsilon:= \sup_X \varphi_\varepsilon$ and $\psi_\varepsilon:= \varphi_\varepsilon - M_\varepsilon$. The equation \eqref{dmae-2-e1} is then  rewritten as
\begin{equation}
\label{dmae-2-e50}
	(\beta + dd^c \psi_\varepsilon)^n = e^{\varepsilon \psi_\varepsilon + \varepsilon M_\varepsilon} f \omega^n.
\end{equation}
Put $\tilde A := \int_X f\omega^n>0$, and recall that we have normalised $\int_X \beta^n =1$. Therefore, by the Stokes theorem
\[
	1 = \int_X (\beta + dd^c \psi_\varepsilon)^n
	   = \int_X e^{\varepsilon \psi_\varepsilon + \varepsilon M_\varepsilon} f \omega^n.
\]
An application of the Jensen inequality gives 
\begin{equation}
\label{dmae-2-e5}
	\log \frac{1}{\tilde A} \geq \int_X (\varepsilon \psi_\varepsilon + \varepsilon M_\varepsilon) \frac{f \omega^n}{\tilde A}
=	\varepsilon M_\varepsilon + \int_X \varepsilon \psi_\varepsilon \frac{f \omega^n}{\tilde A}.
\end{equation}
As $\beta \leq C \omega$ for some $C>0$ and $\sup_X \psi_\varepsilon =0$, it follows from Corollary~\ref{l1-uni-bound} that there exists $C= C(\tilde A, \|f\|_p, X, \omega)>0$ such that
\[
	\int_X \psi_\varepsilon \frac{f \omega^n}{\tilde A} \geq - C.
\]
Therefore, the inequality \eqref{dmae-2-e5} implies that
\[
	\varepsilon M_\varepsilon 
	\leq \varepsilon C + \log\frac{1}{\tilde A} 
	\leq C'
\]
as $0< \varepsilon \leq 1$. This inequality tells us that 
\[
	(\beta + dd^c \psi_\varepsilon)^n \leq e^{C'} f\omega^n.
\]
By the volume-capacity inequality (Proposition~\ref{vol-cap-ine}) and \cite[Theorem 2.1]{EGZ09}, after a simple application of H\"older inequality, we get that 
\begin{equation}
\label{dmae-2-e60}
	- H \leq \psi_\varepsilon \leq 0
\end{equation}
where $H$ depends only on $X, \omega, \beta, \|f\|_p$. Therefore, the right hand side of the equation \eqref{dmae-2-e50} is uniformly bounded in $L^p$, $p>1$.
It follows the proof of  \cite[Proposition 3.3]{EGZ09} or \cite[Corollary 3.4]{EGZ11} that 
\begin{equation}
\label{dmae-2-e6}
	\|\psi_\varepsilon - \psi_\delta\|_\infty \leq C \|\psi_\varepsilon - \psi_\delta\|_1^\gamma
\end{equation}
where $0< \gamma< \gamma_{max}= \gamma_{max}(n,p)$ (explicit formula is given in \cite{EGZ09, EGZ11}). Since $\{\psi_\varepsilon\}_{0< \varepsilon \leq 1}$ is compact in $L^1(\omega^n)$, there exists a subsequence $\{\psi_{\varepsilon_j}\}$, $\varepsilon_j \searrow 0$ as $j \to + \infty$, which   is Cauchy in $L^1(\omega^n)$. By \eqref{dmae-2-e6} this is also a Cauchy sequence in $C(X)$. Therefore, it converges to $\psi\in PSH(\beta) \cap C(X)$ with $\sup_X \psi =0$. We may suppose that $\varepsilon_j M_j$ converges to $M$ as $j \to + \infty$. Thus, by the Bedford-Taylor convergence theorem we have
\[
	(\beta + dd^c \psi)^n = e^M f \omega^n
\]
as $\varepsilon_j \psi_{\varepsilon_j}(x) \to 0$ for every $x\in X$ (see \eqref{dmae-2-e60}). 
It is obvious that 
\[
	e^M = \frac{\int_X \beta^n}{\int_X f\omega^n}:= c \quad 
	\mbox{is uniquely defined.}
\]
Furthermore, there is at most one $\varphi\in PSH(\beta)\cap C(X)$ with $\sup_X\varphi =0$ solving the equation 
\[
	(\beta + dd^c\varphi)^n = c f \omega^n.
\]
It implies that actually
$e^{\varepsilon \varphi_\varepsilon(x)} \to c$, and $\psi_\varepsilon$ converges uniformly to $\psi\equiv \varphi$ as $\varepsilon \to 0$.
\end{proof}

The hard part is to prove the existence of solution in Theorem~\ref{existence-dmae}. We follow the ideas in \cite{EGZ09, EGZ11}.  At the first step we show that the equation admits a solution in the finite self energy class, i.e. the class $\cE^1(X, \beta)$ (see below for the definition). Once the first step is done then the second step is showing that this solution is bounded. It follows from the works of Ko\l odziej \cite{kolodziej98, kolodziej05} and Eyssidieux-Guedj-Zeriahi \cite{EGZ09}. The last step is to show the solution is continuous. It is an immediate consequence of the viscosity approach due to Eyssidieux-Guedj-Zeriahi \cite{EGZ11}. Since the second and last step are well-known, we will not reproduce them here.  We refer the reader to  \cite{kolodziej05, EGZ09, EGZ11} for the details. Thus, we  only focus on the proof of the first step.

Following \cite{GZ07, EGZ09} we say that a $\beta$-psh function $u$ belongs to $\cE^1(X, \beta)$  if there exists a sequence $u_j \in 	 PSH(\beta) \cap L^\infty(X)$ satisfying
\[
	u_j \searrow u \quad \mbox{ and } \quad
	\quad \sup_j \int_X |u_j| (\beta + dd^c u_j)^n < + \infty.
\]
We are in a  position to state our main result in this section.

\begin{theorem}
\label{existence-dmae-1}
Let $\beta$ be a smooth closed semi-positive $(1,1)$-form such that
$
	\int_X \beta^n =1.
$
Let $0\leq f \in L^p(\omega^n)$, $p>1$, be such that 
$
	\int_X f\omega^n>0.
$
Then, there exists a 
function $\varphi \in \cE^1(X,\beta)$ satisfying
\[
	(\beta + dd^c \varphi)^n = e^\varphi f \omega^n
\]
in the weak sense of currents.
\end{theorem}

The outline of the proof is as follows. We approximate the degenerate form $\beta$ by a family of Hermitian metrics $\beta + \varepsilon \omega$, $0<\varepsilon \leq 1$. For each $0< \varepsilon \leq1$, by Theorem~\ref{existence-ma}, there is a unique
$\varphi_\varepsilon \in PSH(\beta + \varepsilon \omega) \cap C(X)$ such that
\begin{equation}
\label{dcmae}
	(\beta + \varepsilon \omega + dd^c \varphi_\varepsilon)^n 
	= e^{\varphi_\varepsilon} f \omega^n.
\end{equation}
Our goal is to show $\{\varphi_\varepsilon\}$ decreases to a solution $\varphi\in \cE^1(X, \beta)$.


The next result shows why the sequence $\{\varphi_\varepsilon\}$ is decreasing 
as $\varepsilon \searrow 0$.
\begin{lemma}
\label{decreasing-property-2} 
Let $\omega, \tilde \omega$ be two Hermitian metrics on $X$ such that $\tilde\omega \leq  \omega$.  Let $\mu$ be a positive Radon measure.
Suppose that $u \in PSH(\omega) \cap L^\infty(X)$ and 
$v \in PSH(\tilde \omega) \cap L^\infty(X)$ satisfy
\begin{equation}
\label{dcpr-1}
	(\omega + dd^c u)^n  = e^u \mu, \quad
	(\tilde \omega + dd^c v)^n = e^v \mu.
\end{equation}
Then, $v \leq u$.
\end{lemma}

The proof is similar to the one of 
Lemma~\ref{decreasing-property}. The difference now is the variation of metrics on the left hand 
side. We include its proof here for the sake of completeness.

\begin{proof}
We argue by contradiction. Assume that $\{v > u\}$
is nonempty. Then,
$m= \inf_X (u - v) <0$. Fix an $0< \varepsilon <<1$ to be determined later, 
and we denote $m(\varepsilon) = \inf_X [u - (1- \varepsilon)v]$. It is clear that
\[
	m - \varepsilon \| v \|_\infty \leq m(\varepsilon) \leq 
	m + \varepsilon \| v \|_\infty. 
\]
Let us denote 
\[
	U_\varepsilon(s) = \{ u < (1- \varepsilon) v + m(\varepsilon) +s \}.
\]
Note that $v \in PSH(\omega)$ as $\tilde \omega \leq \omega$.
Applying the modified comparison principle from \cite{cuong-kolodziej13} we get that, 
for $0< s < \varepsilon_0:= \frac{\varepsilon^3}{16B}$,
\begin{equation}
\label{dcpr-e2}
	\int_{U_\varepsilon(s)}
	(\omega + dd^c (1- \varepsilon)v)^n 
\leq 	[1 + \frac{C s}{\varepsilon^n}] \int_{U_\varepsilon(s)} 
	(\omega + dd^cu)^n .
\end{equation}
Since on $\{ u < (1-\varepsilon)v + m(\varepsilon) +s \}$
\[
	(\omega + dd^c (1-\varepsilon)v)^n 
\geq		(1-\varepsilon)^n e^v \mu
\geq		(1-\varepsilon)^n e^{u-m -2 \varepsilon \| v \|_\infty -s} \mu, 
\]
it follows that
\[
	(1-\varepsilon)^n e^{-m- 2 \varepsilon \|v\|_\infty - s}
	\int_{U_\varepsilon(s)}  e^u \mu
\leq	[1 + \frac{C s}{\varepsilon^n}] \int_{U(s)} e^u \mu.
\]
Note that by  \eqref{dcpr-e2} for every $0< s < \varepsilon_0$
\[
	\int_{U_\varepsilon(s)} e^u \mu
=	\int_{U_\varepsilon(s)} (\omega +dd^cu)^n
	>0.
\]
Since $m <0$, we may choose $0< \varepsilon$ so small that 
\[
	(1-\varepsilon)^n e^{-\frac{m}{2} - 2\varepsilon \|v\|_\infty} > 1 + b
\]
for some $0< b = b(m, \varepsilon)$.  Thus, we get for every $0< s <\min\{ \varepsilon_0,  - \frac{m}{2}\}$ that
\[
	0< b \leq \frac{C s}{\varepsilon^n}.
\]
It is impossible when $s>0$ is small enough. Thus, the proof follows.
\end{proof}

We are ready to prove our theorem.

\begin{proof}[Proof of Theorem~\ref{existence-dmae-1}] Take $\{\varphi_\varepsilon\}$ to be the solutions from \eqref{dcmae}. 
We are going to show that 
\begin{equation}
\label{dmae-e2}
	\varphi_\varepsilon \searrow \varphi \in PSH(\beta) \quad \mbox{ as }\varepsilon \searrow 0
\end{equation}
and
\begin{equation}
\label{dmae-e3}
	\sup_{0<\varepsilon \leq 1} \int_X |\varphi_\varepsilon| 
	(\beta + \varepsilon \omega + dd^c \varphi_\varepsilon)^n < +\infty.
\end{equation}
Let us start with the first property \eqref{dmae-e2}. Applying Lemma~\ref{decreasing-property-2} to $\beta + \varepsilon \omega$ and $\beta+ \varepsilon' \omega$, $0< \varepsilon<\varepsilon'$, we get that the
sequence $\varphi_\varepsilon$ is decreasing as $\varepsilon \searrow 0$. Hence, we can put
\[
	\varphi := \lim_{\varepsilon \rightarrow 0} \varphi_\varepsilon. 
\]
To verify that $\varphi \in PSH(\beta)$, it is enough to show that the sequence $\varphi_\varepsilon$ does not
decrease identically to $- \infty$. More precisely, we will prove that  
\begin{equation}
\label{not-infty}
	\int_X e^{\varphi_\varepsilon} f \omega^n \geq \frac{1}{2}\int_X \beta^n >0 
\end{equation}
for $\varepsilon>0$ small enough. 
In fact, set $M_\varepsilon =\sup_X \varphi_\varepsilon$, and $\psi_\varepsilon = \varphi_\varepsilon - M_\varepsilon$. Then, the equation \eqref{dcmae} reads 
\begin{equation}
\label{ma-e-1}
	(\beta + \varepsilon \omega + dd^c \psi_\varepsilon)^n
=	e^{\psi_\varepsilon + M_\varepsilon} f \omega^n.
\end{equation}

The following estimate  is the most crucial step where the results of \cite{cuong-kolodziej13} play an important role.
\begin{lemma}
\label{growth-phi} 
For $0< \varepsilon \leq 1$
\[
	 \|\psi_\varepsilon \|_\infty
	\leq C \left(\log \frac{1}{\varepsilon}\right)^n
\]
where $C>0$ is a uniform constant. In particular, for each positive integer $k$, 
\[
	\lim_{\varepsilon \rightarrow 0} \varepsilon \|\psi_\varepsilon \|_\infty^k
	=0.
\]
\end{lemma}

\begin{proof}[Proof of Lemma~\ref{growth-phi}] 
In order to use the a priori estimate in \cite{cuong-kolodziej13} one needs to know that
$(\beta + \varepsilon \omega + dd^c \psi_\varepsilon)^n$ is well dominated by the capacity $cap_{\beta + \varepsilon \omega}$.

\begin{claim}
\label{sup-bound}
	$M_\varepsilon$  is uniformly bounded from above. 
\end{claim}

\begin{proof}[Proof of Claim~\ref{sup-bound}] It follows from Proposition~\ref{sup-bound-general} with an observation that the reference metric $\beta+\varepsilon \omega$, $0< \varepsilon \leq 1$, is dominated by  $C \omega$ for some $C>0$. Therefore,
\[
	0<	\alpha_0 (\varepsilon) 
	:= \int_X (\beta + \varepsilon \omega) \wedge e^G\omega^{n-1}
	< C \alpha_0 ,
\]
a uniform constant. Another constant is  
\[
	A:=\int_X f^\frac{1}{n} e^G \omega^n>0 \quad
	\mbox{as } \int_X f\omega^n>0.
\]
Since $0 < \beta + \varepsilon \omega \leq C \omega$, we have 
$\psi_\varepsilon \in PSH(C \omega)$. As $\sup_X \psi_\varepsilon =0$, 
it follows from Corollary~\ref{l1-uni-bound} that
\begin{equation}
\label{l1-bound}
	\int_X |\psi_\varepsilon| f^\frac{1}{n} e^G \omega^n 
\leq  \int_X |\psi_\varepsilon| f^\frac{1}{n} e^G[C\omega]^n 
< 	C'
\end{equation}
with a uniform constant $C'>0$. Proposition~\ref{sup-bound-general} and \eqref{l1-bound} give that
\begin{equation}
\label{sup-boud-e1}
	M_\varepsilon 
\leq 	C' + n \log \frac{\alpha_0(\varepsilon)}{A} < C' + n \log\frac{C\alpha_0}{A}.
\end{equation}
The claim follows.
\end{proof}

We proceed to finish the proof of Lemma~\ref{growth-phi}.
By Claim~\ref{sup-bound} and the equation \eqref{ma-e-1} we get that,
for a uniform constant $C>0$,
\begin{equation}
\label{dominated-measure}
	(\beta + \varepsilon \omega + dd^c \psi_\varepsilon)^n
=	e^{\psi_\varepsilon + M_\varepsilon} f \omega^n
\leq 		e^{C} f \omega^n.
\end{equation}
Here, by the volume-capacity inequality (Proposition~\ref{vol-cap-ine})  and  by  
$
	cap_\beta(E) \leq cap_{\beta + \varepsilon \omega}(E),
$
we get, after using the H\"older inequality,  that for any Borel set $E \subset X$,
\begin{equation}
\label{set-dominated-measure}
	\int_E f \omega^n \leq C \|f\|_p \exp\left(- \frac{a}{cap_\beta^\frac{1}{n}(E)}\right)
\leq 	C' \exp\left(- \frac{a}{ cap_{\beta + \varepsilon \omega}^\frac{1}{n}(E)}\right)
\end{equation}
where $a = a(X, \omega, \beta, p)>0$ and $C'>0$ depends only on $X, \omega, \|f\|_p$.
Let us use the notation 
\[
	\beta_\varepsilon: = \beta + \varepsilon \omega \quad \mbox{and} \quad
	N(\beta_\varepsilon)
:=	\sup\left\{\int_X|\rho|\, \beta_\varepsilon^n: 
	\rho \in PSH(\beta_\varepsilon) \mbox{ and } \sup_X \rho = 0\right\}.
\]
It follows from \eqref{dominated-measure} and \eqref{set-dominated-measure} that
\begin{equation}
\label{ma-exponential-decay}
	(\beta_\varepsilon + dd^c \psi_\varepsilon)^n \leq 
	C' \exp\left(- \frac{a}{cap_{\beta_\varepsilon}^\frac{1}{n}(E)}\right).
\end{equation}
Since $\beta$ is closed and $\beta\geq 0$, the reference Hermitian metric $\beta_\varepsilon$ satisfies
\begin{equation}
\label{epsilon-curva}
\begin{aligned}
	- \frac{B}{\varepsilon} \, (\beta_\varepsilon)^2 \leq 
	2n \,dd^c \beta_\varepsilon \leq \frac{B}{\varepsilon} \, (\beta_\varepsilon)^2, 
\\
 	- \frac{B}{\varepsilon} \, (\beta_\varepsilon)^3 \leq 
	4n^2 \,d \beta_\varepsilon \wedge d^c \beta_\varepsilon \leq \frac{B}{\varepsilon} \, (\beta_\varepsilon)^3.
\end{aligned}
\end{equation}
Thus, the "curvature" constant for $\beta_\varepsilon$ is $B_\varepsilon = \frac{B}{\varepsilon}$.

We wish to apply the {\em a priori} estimate from the proof of \cite[Corollary 5.6]{cuong-kolodziej13} to the Hermitian metric $\beta_\varepsilon$. Because of $B_\varepsilon = \frac{B}{\varepsilon}$ the crucial change is the range of the considered sublevel sets which are close to the minimum of $\psi_\varepsilon$. The value $s>0$ there now can be taken 
\[
	0< s \leq \frac{1}{3}\min\left\{\frac{1}{2^n}, \frac{1}{2^3} \frac{1}{16 B_\varepsilon}\right\} = \frac{\varepsilon}{384 B}
\]
for $\varepsilon>0$ small enough. 

By \eqref{ma-exponential-decay} the Monge-Amp\`ere measure $(\beta_\varepsilon + dd^c \psi_\varepsilon)^n$ satisfies the inequality $(5.2)$ in \cite{cuong-kolodziej13} for the {\em admissible} function
\[
	h(s) = C e^{as} \quad \mbox{ for } \quad a, C>0 \quad \mbox{ independent of } \varepsilon. 
\]
Therefore, the formula in the proof of \cite[Corollary 5.6]{cuong-kolodziej13} gives 
\begin{equation}
\label{growth-e1}
	\| \psi_\varepsilon\|_\infty \leq s_0 + \frac{C N(\beta_\varepsilon)}{\hbar(s_0)}, 
\end{equation}
where we choose 
\begin{equation}
\label{growth-e1-1}
	s_0 = \frac{\varepsilon}{384 B} 
	\geq \varepsilon^{2}
\end{equation}
(for sufficiently small $\varepsilon>0$) and $\hbar (s)$ is the inverse function of 
\[
	\kappa \left ( s^{-n} \right) 		4\, C_n  \left \{
			\frac{1}{ \left [ h ( s )\right]^{\frac{1}{n}} } 
			+ \int_{s}^\infty  \frac{dx}{x \left[ h (x) \right]^{\frac{1}{n}}}		
					\right \}.
\]
The next step is to estimate the right hand side
of \eqref{growth-e1} from above.
By computation we get that
\[
	\kappa(x) \leq C \exp (- \tilde a x^{-\frac{1}{n}})
\]
where $C, \tilde a>0$ are uniform constants. As $\kappa$ is an increasing function, its inverse function satisfies
\begin{equation}
\label{growth-e2}
	\hbar(x) \geq \left (  \frac{1}{\tilde a} \log \frac{C}{x}\right)^{-n}.
\end{equation}
Since $\sup_X \rho =0$ and $\rho \in PSH(C \omega)$, Corollary~\ref{l1-uni-bound}
gives
\[
	\int_X|\rho| (\beta + \varepsilon \omega)^n
\leq  \int_X |\rho| [C\omega]^n < C'.
\]
Therefore,
\begin{equation}
\label{growth-e3}
	N(\beta_\varepsilon) < C.
\end{equation}
Plugging  $s_0 = \varepsilon^{2}$ into the right hand side of \eqref{growth-e1}, then  using \eqref{growth-e2} and \eqref{growth-e3}, we get that
\[
	\| \psi_\varepsilon\|_{L^\infty} \leq
	\varepsilon^{2} + C \left( \frac{1}{\tilde a} \log C - \frac{2 \log \varepsilon}{\tilde a}\right)^n \leq 
	\frac{C}{\tilde a^n} (-\log \varepsilon)^n
\]
for $\varepsilon>0$ small. Thus, Lemma~\ref{growth-phi} follows.
\end{proof}

We are in the position to prove \eqref{not-infty}. Put $\Omega= \beta+ \varepsilon \omega$, then it's easy to see that 
\[
	dd^c \Omega = \varepsilon dd^c \omega, \quad
	d \Omega \wedge d^c \Omega	= \varepsilon^2 d\omega \wedge d^c \omega.
\]
Hence, one can take $B_\Omega=\varepsilon B$ to be the corresponding constant in Proposition~\ref{estimate-mass} for the metric $\Omega$, i.e.
\begin{equation}
\label{B-Omega}
	-\varepsilon B \omega^2 \leq 2n dd^c \Omega \leq \varepsilon B\omega^2, \quad -\varepsilon B \omega^3 \leq 4n^2 d\Omega \wedge d^c \Omega \leq \varepsilon B\omega^3.
\end{equation}
Therefore, 
\begin{equation}
\label{bound-mass-ep}
\begin{aligned}
	\int_X (\beta + \varepsilon \omega)^n -  \varepsilon B \|\psi_\varepsilon\|_{L^\infty}^n C
&\leq 	\int_X (\beta + \varepsilon \omega + dd^c \psi_\varepsilon)^n  \\
&\leq		\int_X (\beta+ \varepsilon \omega)^n +  \varepsilon B\|\psi_\varepsilon\|_{L^\infty}^n C.
\end{aligned}
\end{equation}
Combining Lemma~\ref{growth-phi}  and \eqref{bound-mass-ep} we get the property \eqref{not-infty}.  

Next, we shall see how to get \eqref{dmae-e3}, i.e.
\begin{equation}
\label{sup-mass-e1}
	\sup_{0< \varepsilon \leq 1}
	 \int_X |\varphi_\varepsilon| (\beta + \varepsilon \omega +dd^c \varphi_\varepsilon)^n 
	 < +\infty.
\end{equation}

\begin{claim}
\label{inf-bound-M}
	$M_\varepsilon$ is uniformly bounded from below.
\end{claim}

\begin{proof}[Proof of Claim~\ref{inf-bound-M}]
We use again the equation \eqref{ma-e-1}. It gives us
\[
		(\beta + \varepsilon \omega + dd^c \psi_\varepsilon)^n
=	e^{\psi_\varepsilon + M_\varepsilon} f \omega^n
\leq	e^{M_\varepsilon} f \omega^n.
\]
Integrating two sides 
\[
	\int_X (\beta + \varepsilon \omega + dd^c \psi_\varepsilon)^n
\leq	e^{M_\varepsilon}\int_X f\omega^n.
\]
Combining this, the first inequality of \eqref{bound-mass-ep} and Lemma~\ref{growth-phi}, 
we get
\[
	e^{M_\varepsilon} \int_X f\omega^n
\geq 	\int_X (\beta + \varepsilon \omega)^n - \varepsilon (- \log \varepsilon)^{n^2} C
\geq	\frac{1}{2} \int_X \beta^n >0
\]
for $\varepsilon$ sufficiently small.
Thus, the claim easily follows.
\end{proof}

We proceed to prove \eqref{sup-mass-e1}. By Claims~\ref{sup-bound} and ~\ref{inf-bound-M} there exists a uniform number $M \geq 1$ such that
\[
	- M \leq M_\varepsilon \leq M.
\]
Since $\varphi_\varepsilon = \psi_\varepsilon + M_\varepsilon$, we have 
\begin{align*}
	\int_X |\varphi_\varepsilon| (\beta + \varepsilon \omega + dd^c \varphi_\varepsilon)^n 
&=	\int_X |\psi_\varepsilon + M_\varepsilon| e^{\psi_\varepsilon + M_\varepsilon} f\omega^n 	\\
&\leq 	\int_X (-\psi_\varepsilon) e^{\psi_\varepsilon +M}f \omega^n +  \int_X M e^Mf \omega^n.
\end{align*}
As $\max_{s\leq 0} (-s e^s) = 1/e$, we get \eqref{sup-mass-e1}.  Thus, the second property of the sequence $\varphi_\varepsilon$ is obtained.

We are in a position to show that $\varphi \in \cE^1(X, \beta)$. In other words,
\[
	\int_X |\varphi|(\beta + dd^c \varphi)^n < + \infty
\]
in the sense that there exists a sequence $\varphi_j \downarrow \varphi$, 
$\varphi_j \in PSH(\beta) \cap L^\infty(X)$, satisfying
\begin{equation}
\label{E-1}
	\sup_j \int_X |\varphi_j| (\beta+ dd^c \varphi_j)^n < + \infty.
\end{equation}
By Claims~\ref{sup-bound} and ~\ref{inf-bound-M} we may assume, without lost of generality,  that 
\[
	\varphi_\varepsilon \leq -1
\] 
for every $0< \varepsilon \leq 1$. The same is true for  $\varphi$.
The construction of the sequence $\varphi_j$ is as follows. Let us denote for $j \geq 1$,
\[
	\varphi_j = \max\{ \varphi, - j\}.
\]
As $\beta \geq 0$, $\varphi_j \in PSH(\beta)$ for $j\geq 1$. It's obvious that 
$\varphi_j \searrow \varphi$. We shall check \eqref{E-1}. Let us denote
\[
\varphi_{\varepsilon, j} = \max\{\varphi_\varepsilon , - j\}.
\]
Observe that, if $j$ is fixed, then $\varphi_{\varepsilon, j}$ is uniformly bounded
and it decreases to $\varphi_j \in PSH(\beta) \cap L^\infty(X)$ as $\varepsilon \searrow 0$. Hence, by the Bedford-Taylor convergence theorem \cite{BT82} we have 
\[
	(\beta + \varepsilon \omega + dd^c \varphi_{\varepsilon,j})^n 
\rightarrow (\beta + dd^c \varphi_j)^n \quad \mbox{ weakly \; as } \varepsilon \rightarrow 0.
\]
We will verify that for a fixed $j$,
\begin{equation}
\label{e1-1}
	 \int_X (- \varphi_{\varepsilon, j}) (\beta + \varepsilon \omega + dd^c 
	\varphi_{\epsilon, j})^n < C,
\end{equation}
where $C>0$ is a uniform constant independent of $j, \varepsilon$. 
It follows from \eqref{B-Omega} that $B_\Omega= \varepsilon B$ is the constant in \eqref{curvature-Omega} for the metric
$
	\Omega = \beta + \varepsilon \omega.
$
Hence, Proposition~\ref{in-e1-class} applied to 
$
	u = \varphi_\varepsilon  \leq v= \varphi_{\varepsilon,j} \leq -1
$
implies that
\begin{align*}
	\int_X (- \varphi_{\varepsilon,j}) (\beta+ \varepsilon \omega + dd^c \varphi_{\varepsilon,j})^n 
&\leq		2^n\int_X (-\varphi_\varepsilon) (\beta + \varepsilon \omega + dd^c \varphi_\varepsilon)^n +	\\
&\quad 	+ \varepsilon B  (-\log \varepsilon)^{4n^2} C,
\end{align*}
where the last term on the right hand side comes from Lemma~\ref{growth-phi}:
\[
	|\varphi_{\varepsilon,j}| \leq |\varphi_\varepsilon| \leq |\psi_\varepsilon| + M
	\leq C(-\log \varepsilon)^n.
\]
Therefore, for $0< \varepsilon \leq 1$,
\[
\int_X (- \varphi_{\varepsilon,j}) (\beta+ \varepsilon \omega + dd^c \varphi_{\varepsilon,j})^n
\leq		2^n \int_X (-\varphi_\varepsilon) (\beta + \varepsilon \omega + dd^c \varphi_\varepsilon)^n	+	C.
\]
By \eqref{sup-mass-e1} the right hand side is under control.
Therefore, we get \eqref{e1-1}.

We are ready to justify \eqref{E-1}, it is a classical argument. Since $\varphi_{\varepsilon,j}$
is u.s.c, any limit point $\nu$ of $\{ -\varphi_{\varepsilon, j} (\beta +\varepsilon \omega + dd^c \varphi_{\varepsilon,j})^n\}_{\varepsilon>0}$ as $\varepsilon \searrow 0$ satisfies 
\[
	0	\leq  (-\varphi_j) (\beta + dd^c \varphi_j)^n \leq \nu.
\]
It follows that
\[
	0\leq \int_X (-\varphi_j) (\beta + dd^c \varphi_j)^n
\leq		\liminf_{\varepsilon \rightarrow 0} 
		\int_X (-\varphi_{\varepsilon,j}) (\beta_\varepsilon + dd^c \varphi_{\varepsilon,j})^n < C < + \infty
\]
where the last inequality is by \eqref{e1-1}. Thus, we proved \eqref{E-1}, i.e.  
$\varphi \in \cE^1(X,\beta)$.

\medskip
{\em End of the proof of Theorem~\ref{existence-dmae-1}.} It remains to show that $\varphi$ satisfies the Monge-Amp\`ere equation. Since $M \geq \varphi_\varepsilon \searrow \varphi$ as $\varepsilon \searrow 0$,  it follows from the  Lebesgue dominated convergence  theorem
that
\[
	e^{\varphi_\varepsilon} f \omega^n \rightarrow e^\varphi f \omega^n.
\]
By  \eqref{dcmae} we will finish if we can show
\begin{equation}
\label{weak-convergence}
	(\beta + \varepsilon \omega + dd^c \varphi_\varepsilon)^n
	\rightarrow (\beta  + dd^c \varphi)^n \quad \mbox{ as } \varepsilon \rightarrow 0.
\end{equation}
Indeed, this follows from the fact that $(\beta + \varepsilon \omega + dd^c \varphi_{\varepsilon, j})^n
\rightarrow (\beta+ dd^c \varphi_j)^n$ as $\varepsilon \rightarrow0$ for any fixed $j$,
and
\[
	\int_{\{\varphi_\varepsilon \leq - j\}} 
	(\beta + \varepsilon \omega + dd^c \varphi_\varepsilon)^n 
\leq 	\frac{1}{j} \int_X |\varphi_\varepsilon| 
	(\beta + \varepsilon \omega + dd^c \varphi_\varepsilon)^n \leq \frac{C}{j},
\]
where $C>0$ is independent of $\varepsilon$ and $j$ (see \eqref{sup-mass-e1}). 
The proof of the theorem is completed.
\end{proof}

Thus, we also proved Theorem~\ref{existence-dmae} by using the  outline above and Theorem~\ref{existence-dmae-1}. 

\bigskip

As in the K\"ahler setting, from Theorems~\ref{existence-dmae}, ~\ref{existence-ma} and Lemma~\ref{decreasing-property-2} we get the following result.

\begin{corollary}
\label{existence-dmae-3}
Let $0\leq f \in L^p(\omega^n)$, $p>1$, be such that 
$\int_X f\omega^n>0$. Let $\varphi\in PSH(\beta)\cap C(X)$ be the unique solution
to 
\[
	(\beta + dd^c \varphi)^n = e^\varphi f \omega^n.
\]
For $0< \varepsilon \leq 1$ let $\varphi_\varepsilon \in PSH(\beta + \varepsilon \omega) \cap C(X)$ be the unique solution to
\[
	(\beta + \varepsilon \omega + dd^c \varphi_\varepsilon)^n 
	= e^{\varphi_\varepsilon} f \omega^n.
\]
Then, $\varphi_\varepsilon \searrow \varphi$ uniformly as $\varepsilon \searrow 0$.
\end{corollary}

\bigskip

\section{Applications}
\label{Sect-4}

In this section we consider applications of results in Sections ~\ref{Sect-2} and \ref{Sect-3} to several problems on compact Hermitian manifolds. 
Recall that we use the normalization 
\[
	d^c = \frac{i}{2\pi} (\bar \partial - \partial), \quad 
	dd^c =\frac{i}{\pi} \partial \bar\partial.
\]
We consider the real Bott-Chern cohomology 
group
\[
	H^{1,1}_{BC} (X, \bR) 
	= \frac{\{\mbox{closed real } (1,1)\mbox{-forms}\}}
	{\{ dd^c \phi, \phi \in C^\infty (X, \bR)\}}
\]
Let $\{\beta\}$ denote the class of $\beta$ in $H^{1,1}_{BC} (X, \bR)$. The class  
$\{\beta\}$ is called {\em nef} if for any $\varepsilon>0$ there exists a representative 
$\beta + dd^c \psi$ such that
\[
	\beta + dd^c \psi > - \varepsilon \omega,	
\] 
where $\omega$ is some fixed Hermitian metric on $X$. 

If $\beta \geq 0$, then $\{\beta\}$ is a {\em nef} class. If moreover it has positive highest self-intersection number $\int_X \beta^n>0$, then these assumptions allow us to solve the degenerated Monge-Amp\`ere equation with the background form being $\beta$. This gives a way to construct non-trivial $\beta$-psh functions on $X$. We will get some interesting consequences of this fact.

\subsection{The Tosatti-Weinkove and Demailly-Paun conjectures}
We shall start with a verification of a conjecture of Tosatti and Weinkove \cite{TW12a} under the assumption that $\beta \geq 0$. 
Let us state the result.

\begin{theorem}
\label{tw-conjecture}
Let $X$ be a $n$-dimensional compact complex manifold. Suppose there exists a 
class $\{\beta\} \in H^{1,1}_{BC} (X, \bR)$ which is semi-positive and satisfies 
$\int_X \beta^n >0$. Let $x_1, ..., x_N \in X$ be fixed points and let
$\tau_1, ..., \tau_N$ be positive real numbers so that 
\begin{equation}
\label{tw-assumption}
	\sum_{j=1}^N \tau_j^n < \int_X \beta^n.
\end{equation}
Then there exists a $\beta$-plurisubharmonic function $\varphi$ with logarithmic
poles at $x_1,..., x_N$:
\[
	\varphi(z) \leq \tau_j \log |z| + O(1),
\]
in a coordinate neighbourhood $(z_1, ..., z_n)$ centered at $x_j$, where 
$|z|^2 = |z_1|^2+...+|z_n|^2$.
\end{theorem}

\begin{remark}
\label{tw12-facts} The conjecture in \cite{TW12a} is stated for $\beta$ being {\em nef} and $\int_X \beta^n>0$, which is motivated by the corresponding result of Demailly on K\"ahler manifolds. We refer the reader to \cite{demailly93} for applications in algebraic geometry. Tosatti and Weinkove proved their conjecture for $n=2,3$ and obtained some partial results for general $n$ (if $X$ is Moishezon and $\{\beta\}$ is a rational class).
\end{remark}


\begin{proof}[Proof of Theorem~\ref{tw-conjecture}] Given Theorem~\ref{existence-dmae-2}, Demailly's mass concentration technique in \cite[Sec. 6]{demailly93} is adaptable immediately (see also \cite{TW12a}). 
For the sake of completeness we repeat it here.
We choose coordinates $(z_1,..., z_n)$ in a neighbourhood centered at $x_j$.
Let $\chi : \bR \rightarrow \bR$ be a smooth, convex, increasing, which satisfies
$\chi(t) = t$ for $t\geq 0$, and $\chi(t) = -\frac{1}{2}$ for $t\leq -1$. Put, for $\varepsilon>0$,
\[
	\gamma_{j, \varepsilon} 
	= dd^c \left( \chi\left(\log\frac{|z|}{\varepsilon}\right)\right),
\]
where $|z|^2 = |z_1|^2 + ...+ |z_n|^2$. As observed in \cite{demailly93} 
$\gamma_{j,\varepsilon}$ is a closed $(1,1)$-positive form on this coordinate chart,
and it is equal to $dd^c \log|z|$ outside the ball $\{|z| \leq 
\varepsilon\}$. Then, we may extend $\gamma_{j, \varepsilon}^n = 0$ for $|z|>\varepsilon$.
Thus, $\gamma_{j, \varepsilon}^n$ is a smooth non-negative $(n,n)$-form on $X$
satisfies
\[
	\int_X \gamma_{j, \varepsilon}^n =1,
\]
and $\gamma_{j, \varepsilon}^n \rightarrow \delta_{x_j}$ the Dirac  measure mass 
at $x_j$ as $\varepsilon \rightarrow 0$. Put 
\[
	\delta = \int_X \beta^n - \sum_{j=1}^n \tau_j^n > 0.
\]
Using  Theorem~\ref{existence-dmae-2} we solve the Monge-Amp\`ere equation
\[
\left( \beta +  dd^c \varphi_\varepsilon \right)^n 
= \sum_{j=1}^N \tau_j^n  \gamma_{j,\varepsilon}^n + \delta \frac{\omega^n}{\int_X\omega^n} 
\]
where 
$
	\varphi_\varepsilon \in PSH(\beta) \cap C(X), 
	\sup_X  \varphi_\varepsilon =0. 
$ 
The family $\{\varphi_\varepsilon: \sup_X \varphi_\varepsilon =0\}_{\varepsilon>0}$ is compact in $L^1(\omega^n)$. 
Then, there exists a subsequence $\varepsilon \rightarrow 0$ such
that $\varphi_\varepsilon $ converges to  a $\varphi \in PSH(\beta)$ in $L^1(\omega^n)$. 

We show now that the funciton $\varphi$ has desired singularities. 
Let
$U$ be a neighbourhood of $x_j$ and suppose that 
\[
	\beta = dd^c h
\]
for $h\in C^\infty(\overline{U})$ on $\overline{U}$. Set 
$v:= \psi_\varepsilon = h+ \varphi_\varepsilon$. Since $h\in C^\infty(\overline{U})$, there exists a uniform constant $C$ such that $v_{|_{\partial U}} \leq C$.  Consider 
\[
	u = \tau_j \left(\chi(\log\frac{|z|}{\varepsilon}) + \log \varepsilon\right)
		+ C_1 ,
\]
where $C_1$ is a large constant. Then for $\varepsilon>0$ small enough
\begin{equation*}
\begin{aligned}
	u_{|_{\partial U}} = \tau_j \log|z| + C_1, 
	\quad v_{|_{\partial U}} \leq C, \\
	(dd^c v)^n \geq \tau_j^n \gamma_\varepsilon^n = (dd^c u)^n \quad 
	\mbox{on } U.
\end{aligned}
\end{equation*}
For $C_1$ sufficiently large, $u\geq v$ on $\partial U$, then it follows from the Bedford-Taylor comparison principle \cite{BT82} that $u \geq v$ on $U$. Hence,
\[
	\psi_\varepsilon \leq \tau_j \log(|z| + \varepsilon) + C_2 \quad
	\mbox{on } U. 
\]
Thus, $\varphi(z) \leq \log|z| + O(1)$ in $U$.
\end{proof}

We turn now to a weak version of a conjecture of Demailly and Paun \cite[Conjecture 0.8]{DP04}. Suppose that a compact $n$-dimensional complex manifold possesses a semi-positive cohomology class $\{\beta\}$ of type $(1,1)$ such that $\int_X \beta^n>0$, then it is conjectured that this manifold belongs to the Fujiki class $\cC$. We are able to prove this for manifolds equipped with a pluriclosed metric. Let us state our result.

\begin{theorem}
\label{dp-semipositive}
Let $(X, \omega)$ be a $n$-dimensional compact complex manifold equipped with the pluriclosed metric $\omega$, i.e. $dd^c \omega =0$. Assume that $X$ possesses a closed semi-positive cohomology class $\{\beta\}$ of type $(1, 1)$ such that 
$\int_X \beta^n > 0$. Then $\{\beta\}$ contains a K\"ahler current $T$, i.e. $T\geq \delta \omega$ for some $\delta >0$.
\end{theorem}

\begin{remark}
\label{chiose-result} In the case of complex surfaces, i.e. $n=2$, there always exits a Gauduchon (pluriclosed) metric and the theorem is known thanks to the work of N. Buchdahl \cite{buchdahl99, buchdahl00} and A. Lamari \cite{lamari99a, lamari99b}. When $n=3$ as pointed out by Chiose \cite[Remmark 3.3]{chiose13} that the semi-positive assumption of $\{\beta\}$ in Theorem~\ref{dp-semipositive} can be weaken to  the {\em nef} one. Then, Theorem~\ref{dp-semipositive} is only interesting for $n\geq 4$. 
\end{remark}

\begin{remark} By \cite[Theorem 0.7]{DP04} we know that a compact complex manifold which carries a K\"ahler current is in the Fujiki class. 
Recently, Chiose \cite[Theorem 0.2]{chiose14} shows that if $X$ 
belongs to the Fujiki class and possesses a pluriclosed metric, then $X$ is indeed a K\"ahler manifold.
\end{remark}

Our arguments follow the ideas of Chiose \cite{chiose13} who used the  non-degenerate Monge-Amp\`ere equation in \cite{TW10b} to give a simpler  proof of Demailly-Paun's theorem \cite[Theorem 2.12]{DP04}. 
Later on, Popovici \cite[Lemma 3.1]{popovici14} made an observation that can help to simplify some  arguments in \cite{chiose13}. We will also make use of this observation.

Instead of using non degenerate Monge-Amp\`ere equation as in \cite{DP04} \cite{chiose13} we use the one with the degenerate left hand side (Section~\ref{Sect-3}). We do not add a small positive Hermitian metric to the degenerate metric $\beta\geq 0$. The advantage is that the total mass of the left hand side is invariant (=$\int_X \beta^n$). However, {\em a priori} the obtained solutions are only continuous, so we need to extend \cite[Lemma 3.2]{popovici14} to the singular $(1,1)$-forms setting. This also required results in previous sections.

 First, we recall a useful lemma due to Lamari \cite{lamari99a}.

\begin{lemma}
\label{la99}
Let $\alpha$ be a smooth real $(1, 1)$ form. There exists a distribution $\psi$ on $X$ such that $\alpha + dd^c \psi \geq 0$ if and only if
\[
	\int_X \alpha \wedge \gamma^{n-1} \geq 0
\] 
for any Gauduchon metric $\gamma$, i.e. $dd^c \gamma^{n-1} =0$, on $X$.
\end{lemma}

\begin{proof}[Proof of Theorem~\ref{dp-semipositive}] 
We argue by contradiction. Suppose that $\{\beta\}$ is not a K\"ahler current. This means there is a sequence $\delta_j \downarrow 0$, $j\geq 1$, and
there is no representative $\beta + dd^c u$, $u \in \cD(X)$ such that the real  
$(1,1)$ current $\beta + dd^c u - \delta_j\omega$ is positive. By Lemma~\ref{la99} it implies that for every $j \geq1$ there exists a Gauduchon metric $g_j$ such that
\[
	\int_X (\beta - \delta_j \omega) \wedge g_j^{n-1}
	\leq 0.
\]
Put $G_j:= g_j^{n-1}$. Then, the inequality is equivalent to
\begin{equation}
\label{dp-ineq-j-1}
	\int_X \beta \wedge G_j
	\leq \delta_j \int_X \omega \wedge G_j.
\end{equation}
Since $\beta \geq 0$ and $\int_X \beta^n>0$,
using Theorem~\ref{existence-dmae-2} we solve, for $j\geq 1$,
\begin{equation}
\label{dp-ma1}
	(\beta + dd^c v_j)^n 
	= c_j \omega \wedge G_j, \quad
	v_j \in PSH(\beta) \cap C(X), \quad
	\sup_X v_j =0.
\end{equation}
Here, by the Stokes theorem,
\begin{equation}
\label{dp-eq-c-j}
	c_j = \frac{\int_X (\beta + dd^c v_j)^n}{\int_X \omega \wedge G_j} 
	= \frac{\int_X \beta^n}{\int_X \omega \wedge G_j}
	>0.
\end{equation}
Put 
$
	\beta_j = \beta+ dd^c v_j.
$
It is easy to see that
\begin{equation}
\label{dp-ineq-j}
	 \int_X  \beta_j \wedge G_j  = \int_X \beta \wedge G_j
	\leq \delta_j \int_X \omega \wedge G_j.
\end{equation}
Next, we are going to prove that
\begin{equation}
\label{trace-estimate}
	\int_X \beta_j \wedge G_j \cdot \int_X \beta_j^{n-1} \wedge \omega
\geq	\frac{c_j}{n} \left(\int_X \omega \wedge G_j\right)^2.
\end{equation}
We prove it by reducing to the case $\beta_j$ is smooth and positive definite (see \cite[Lemma 3.2]{popovici14}). Let us fix $j$ for a moment as it does not affect our proof. We remark that the reducing process uses in an essential way results in Sections~\ref{Sect-2} and ~\ref{Sect-3}.

\begin{lemma}
\label{step1-tre}
For $0< \varepsilon \leq 1$, let  $v_\varepsilon \in PSH(\beta) \cap C(X)$ be the unique solution to 
\begin{equation}
\label{dp-ma2}
	(\beta + dd^c v_\varepsilon)^n 
= 	e^{\varepsilon v_\varepsilon} \omega \wedge G_j
\end{equation}
(by Theorem~\ref{existence-dmae}). Then, 
\begin{equation}
\label{eq1-step1}
	\int_X (\beta + dd^c v_\varepsilon) \wedge G_j  \cdot
		\int_X (\beta + dd^c v_\varepsilon)^{n-1} \wedge \omega
\geq \frac{1}{n} 
		\left( \int_X e^\frac{\varepsilon v_\varepsilon}{2} \omega \wedge G_j \right)^2.
\end{equation}
\end{lemma}

\begin{proof}[Proof of Lemma~\ref{step1-tre}]
Since $0<\varepsilon \leq 1$ is fixed, by rescaling we work with $\tilde\beta:= \varepsilon \beta$, $\tilde \omega = \varepsilon \omega$, $\tilde G_j := \varepsilon^{n-1} G_j$ and $\tilde v:=\varepsilon v_\varepsilon$ instead of the original ones.  Thus, without lost of generality, we may assume $\varepsilon =1$ and write $u:= v_\varepsilon$. As $j$ is fixed, to avoid confusion of notations later,
in the proof of this lemma we write 
\begin{equation}
\label{step1-tre-e0}
	 g:= g_j \quad \mbox{and} \quad G:=g^{n-1}= G_j.
\end{equation}
Then, $(\beta + dd^c u)^n = e^u \omega \wedge G$ and we need 
to show that
\begin{equation}
\label{step1-tre-e1}
	\int_X (\beta + dd^c u) \wedge G \cdot
		\int_X (\beta + dd^c u)^{n-1} \wedge \omega
\geq \frac{1}{n} 
		\left( \int_X e^\frac{u}{2} \omega \wedge G \right)^2.
\end{equation}
By a theorem of Cherrier \cite[Th\'eor\`em~1, p.373]{cherrier87}, for any $s\geq 1$, we solve $u_s \in PSH(\beta + \frac{1}{s} \omega) \cap C^\infty(X)$  satisfying
\begin{equation}
\label{ma-appr-tr-es}
	\left(\beta + \frac{1}{s} \omega + dd^c u_s \right)^n 
	= e^{u_s} \omega \wedge G
\end{equation}
with 
\[
	\tau_s:= \beta + \frac{1}{s} \omega + dd^c u_s >0.
\]
We know from Corollary~\ref{existence-dmae-3} that $u_s \searrow u \in PSH(\beta) \cap C(X)$ as $s \to +\infty$. Since $G$ is fixed, it follows from the Bedford-Taylor  convergence theorem \cite{BT82} that 
$\tau_s \wedge G \rightarrow (\beta + dd^c u) \wedge G$ 
and $\tau_s^{n-1}\wedge \omega \rightarrow (\beta + dd^c u)^{n-1} \wedge \omega$ 
weakly as $s \to +\infty$. That means
\[
	\int_X (\beta + dd^c u) \wedge G  \cdot
	\int_X (\beta + dd^c u)^{n-1} \wedge \omega
=	\lim_{s \to + \infty} 
	\int_X \tau_s \wedge G  \cdot \int_X \tau_s^{n-1} \wedge \omega.
\]

\begin{claim}
\label{tr-es-sm}
$
\int_X \tau_s \wedge G  \cdot \int_X \tau_s^{n-1} \wedge \omega
\geq 		\frac{1}{n} \; \left( \int_X \sqrt{\frac{\tau_s^n}{\omega^n} \frac{\omega\wedge G}
		{\omega^n}} \; \omega^n \right)^2.
$
\end{claim}


\begin{proof}[Proof of Claim~\ref{tr-es-sm}]
Since the datum are smooth, we write 
\[
	\tau_s \wedge G = \frac{\tau_s \wedge G}{\omega^n} \; \omega^n,
	\quad
	\tau_s^{n-1} \wedge \omega = \frac{\tau_s^{n-1} \wedge \omega}{\omega^n} \; \omega^n.
\]
It follows from the Cauchy-Schwarz inequality that
\[
	\int_X \tau_s \wedge G \cdot \int_X \tau_s^{n-1} \wedge \omega
\geq \left( \int_X \sqrt{\frac{\tau_s \wedge G}{\omega^n} \cdot \frac{\tau_s^{n-1} \wedge \omega}{\omega^n}} \; \omega^n \right)^2.
\]
As $G= g^{n-1}$ (see \eqref{step1-tre-e0}), to get Claim~\ref{tr-es-sm} it's enough to verify that
\[
	\frac{\tau_s \wedge g^{n-1}}{\omega^n} \cdot \frac{\tau_s^{n-1} \wedge \omega}{\omega^n}
\geq \frac{1}{n} \; \frac{\tau_s^n}{\omega^n} \cdot \frac{\omega\wedge g^{n-1}}{\omega^n}.
\]
This inequality is equivalent to
\[
	\frac{\tau_s \wedge g^{n-1}}{g^n} \cdot \frac{g^n}{\omega^n} \cdot\frac{\tau_s^{n-1} \wedge \omega}{\tau_s^n} \cdot \frac{\tau_s^n}{\omega^n}
\geq	 \frac{1}{n} \; \frac{\tau_s^n}{\omega^n} \cdot \frac{\omega\wedge g^{n-1}}{\omega^n}.
\]
Eliminating $\frac{\tau_s^n}{\omega^n}$ on both 
sides  and using the definition of trace, we reformulate it as
\begin{equation}
\label{eq1-cl-tr-es-sm}
	\frac{1}{n^2} \; (tr_{g} \tau_s)(tr_{\tau_s} \omega) \cdot \frac{g^n}{\omega^n}
\geq  \frac{1}{n} \; \frac{\omega\wedge g^{n-1}}{\omega^n}.
\end{equation}
By \eqref{eq3-cl-tr-es-sm} below, we get that
\[
	\frac{1}{n^2} \; (tr_{g} \tau_s)(tr_{\tau_s} \omega) \cdot \frac{g^n}{\omega^n}
\geq \frac{1}{n^2} \; tr_{g} \omega \cdot \frac{g^n}{\omega^n} 
=	\frac{1}{n} \; \frac{\omega \wedge g^{n-1}}{\omega^n}.
\]
Thus, \eqref{eq1-cl-tr-es-sm} is proved. It is left to prove 
\begin{equation}
\label{eq3-cl-tr-es-sm}
	(tr_{g} \tau_s) (tr_{\tau_s} \omega)  
\geq  tr_{g} \omega.
\end{equation}
This is a point-wise inequality. At $O \in X$, we choose a local
coordinate such that
\[
	g(O) = \sum dz_l \wedge d\bar{z}_l, \quad
	\tau_s(O) = \sum a_l dz_l \wedge d\bar{z}_l, \quad
	\omega(O) =	\sum b_{lm} dz_l \wedge d\bar{z}_m,
\]
where $a_l>0$ and the Hermitian matrix $(b_{lm})$ is positive definite.
Therefore,
\[
	tr_{g}\tau_s = \sum a_l, \quad
	tr_{\tau_s} \omega = \sum \frac{b_{ll}}{a_l},\quad
	tr_{g}\omega = \sum b_{ll}.
\]
Hence, \eqref{eq3-cl-tr-es-sm} follows by an elementary calculation.
Claim~\ref{tr-es-sm} is proved.
\end{proof}
It follows from Claim~\ref{tr-es-sm} and \eqref{ma-appr-tr-es} that
\[
	\int_X \tau_s \wedge G \cdot \int_X \tau_s^{n-1} \wedge \omega
\geq		\frac{1}{n} \left( \int_X e^\frac{u_s}{2} \omega \wedge G \right)^2.
\]
Let $s \to +\infty$, we get \eqref{step1-tre-e1}. We finished the proof of Lemma~\ref{step1-tre}.
\end{proof}

We are ready for the verification of \eqref{trace-estimate}. By \eqref{dmae-2-e2} in
Theorem~\ref{existence-dmae-2} for fixed $j$ and fixed $x \in X$ we  have
\[
	c_j = \lim_{\varepsilon \to 0} e^{\varepsilon v_\varepsilon (x)}
\]
where $c_j$ and $v_\varepsilon$ are from \eqref{dp-ma1} and
\eqref{dp-ma2}, respectively. It also follows from Theorem~\ref{existence-dmae-2} that
\[
	v_\varepsilon  - \sup_X v_\varepsilon \rightarrow v_j \quad
	\mbox{ uniformly as } \varepsilon \searrow 0.
\]
Since the left hand side of \eqref{eq1-step1} does not change if we replace
$v_\varepsilon$ by $v_\varepsilon - \sup_X v_\varepsilon$, letting 
$\varepsilon \to 0$ gives us
\[
	\int_X \beta_j \wedge G_j \cdot \int_X \beta_j^{n-1} \wedge \omega
\geq	\frac{c_j}{n} \left(\int_X \omega \wedge G_j \right)^2.
\]
Combining this with \eqref{dp-eq-c-j} and \eqref{dp-ineq-j} we get that
\[
	\delta_j\int_X \beta_j^{n-1} \wedge \omega \geq \frac{\int_X \beta^n}{n}.
\]
We use now our assumption on $\omega$.
Since $dd^c \omega =0$, the Stokes theorem gives 
\[
	\int_X \omega \wedge \beta_j^{n-1} 
	= \int_X \omega \wedge (\beta + dd^c v_j)^{n-1}
	= \int_X \omega \wedge \beta^{n-1}
	= C < + \infty.
\]
Therefore,
\[
	\delta_j C \geq \frac{\int_X \beta^n}{n} \quad\mbox{for every} \quad
	 \delta_j \searrow 0.
\]
It is not possible. Thus, the proof is complete.
\end{proof}

\subsection{The Chern-Ricci flow on smooth minimal models of general type}
\label{S4-2}

The Chern-Ricci flow is the analogue of K\"ahler-Ricci flow on Hermitian manifolds.
Thanks to the works of Gill \cite{gill11, gill13} and Tosatti-Weinkove \cite{TW12b, TW12c, TWYang13}, among others,  interesting results on the Chern-Ricci flow have been proved. Those results are often inspired by corresponding results of the K\"ahler-Ricci flow. The authors applied also the Chern-Ricci flow to investigate open problems in Hermitian geometry \cite{TW12c, TWYang13}.
We refer the reader to those papers and references therein for more details on the Chern-Ricci flow and its applications.

We are interested in recent preprint by Gill \cite{gill13} in which he obtained a generalisation of a result of Tsuji \cite{tsuji88} and Tian-Zhang \cite{tian-zhang06}. We wish to use the pluripotential technique to improve his result as it is done in \cite{EGZ09}. Roughly speaking we use the elliptic complex Monge-Amp\`ere equation to improve the results obtained by the parabolic one. It seems to be known for the experts (see Remark~\ref{gill13-egz11} below) but we feel it is worthwhile to write down it here for the record and it is also a consequence of our results.

We first give a setup which is taken from \cite{gill13}. Let $\omega_0$ be a fixed 
Hermitian metric on $X$. Then, the normalised Chern-Ricci has the form
\begin{equation}
\label{ncr-flow}
\begin{cases}
	\frac{\partial }{\partial t} \omega(t) = - Ric(\omega (t)) - \omega(t) \\
	\omega(0) = \omega_0,
\end{cases}
\end{equation}
where $Ric(\omega(t)) = -i \partial \bar \partial \log [\omega(t)]^n$.

We shall investigate the flow on a special class of Hermitian manifolds. 

\begin{definition}
\label{special-hermitian-manifold}
A Hermitian manifold is called smooth minimal model of general type if 
$c_1^{BC}(K_X)$ is {\em nef} and $K_X$ is a {\em big} line bundle.
\end{definition}

On such a manifold there is a singular set which plays an important role
in studying the regularity of the flow \eqref{ncr-flow}. It is defined in \cite{collins-tosatti13}.
\begin{definition}
\label{null-loc}
The null locus of a smooth minimal model of general type is defined to be
\[
	E := 
\bigcup \left\{
	V \subset X: V \mbox{ is a subvariety, } dim V = k,  \; \int_V \left(c_1^{BC}(K_X)\right)^k =0
\right\}.
\]
\end{definition}

Gill has obtained the following 
\begin{theorem}[\cite{gill13}]
\label{gill-13}
Let $X$ be a smooth minimal model of general type and $E$ is its null locus. Then the normalized Chern-Ricci flow has smooth solution for all time $t\geq 0$ and 
\[
	\omega(t) \to \omega_{KE} \quad \mbox{ as } t \to +\infty
\]
in the weak sense, where $\omega_{KE}$ is a closed positive $(1,1)$-current.
Moreover, the convergence is in $C^\infty_{loc}(X \setminus E)$ and 
\[
	Ric(\omega_{KE}) 
	= - \omega_{KE} \quad \mbox{ on } X \setminus E.
\]
\end{theorem}

$E$ is the smallest set that we can take because of a result of 
Collins and Tosatti \cite{collins-tosatti13}. It is a generalisation of  \cite[Theorem 0.5]{DP04} to manifolds in the Fujiki class.

\begin{lemma}[Collins-Tosatti '13]
\label{C-T-13}
There exists $\psi \in L^1(X)$ such that
\[
	\beta + i \partial\bar\partial \psi \geq c_0 \omega_0
\]
for some $c_0>0$. The function $\psi \in C^\infty(X\setminus E)$, where $E$ is the null locus of $X$ and $E =  \{\psi = -\infty\}$.
\end{lemma}

On the projective manifold of general type, it is showed in \cite{EGZ09} that 
the solution constructed by the K\"ahler-Ricci flow \cite{tsuji88}, \cite{tian-zhang06}, coincides with the solution constructed by the elliptic Monge-Amp\`ere equation. Therefore, the potential of $\omega_{KE}$ is continuous \cite{EGZ09, EGZ11}.

Thanks to our results in previous sections and the arguments in \cite{EGZ09} we are able to get the same statement for the Chern-Ricci flow on smooth minimal models of general type. To state our result we need some notation. First, as $X$ is Moishezon, there exists a smooth closed semi-positive $(1,1)$-form $\beta$ such that 
\[
	\{\beta\} = - c_1^{BC}(X).
\]
Moreover, we can pick a Hermitian metric $\Omega$ satisfying
\[
	\beta = i \partial \bar\partial \log \Omega^n, \quad
	\int_X \Omega^n = \int_X \omega_0^n.
\]
As $K_X$ is big we have 
\[
	\int_X \beta^n >0.
\]
\begin{theorem}
\label{gill13-improvement}
The closed positive $(1,1)$-current $\omega_{KE}$ in Theorem~\ref{gill-13} 
has a unique continuous potential, i.e. 
there is a unique continuous function $\varphi$ such that
\[
	(\beta + dd^c \varphi)^n = e^\varphi \Omega^n,
	\quad \mbox{with} \quad
	\omega_{KE} = \beta + dd^c \varphi \geq 0,
\]
in the weak sense on $X$ and in $C^\infty_{loc}(X\setminus E)$.
\end{theorem}

\begin{remark}
\label{gill13-egz11}
Under the assumptions of Theorem~\ref{gill-13}, $X$ belongs to the Fujiki class. Then, the results in \cite{EGZ11} and \cite{collins-tosatti13} can be used to repeat all arguments of \cite{EGZ09}. Thus, we would get another proof of Theorem~\ref{gill13-improvement}.
\end{remark}

Having the above lemma the proof  of Theorem~\ref{gill13-improvement} follows the lines of that of \cite[Proposition 4.3, 4.4]{EGZ09}.

\end{document}